\documentclass[a4paper, 10pt]{amsart}
\usepackage{amsthm}
\usepackage[]{amsmath}
\usepackage{amssymb}
\usepackage{enumerate}
\usepackage{tabularx}
\usepackage[]{color}
\usepackage[left=3cm, right=3cm]{geometry}
\usepackage[colorlinks]{hyperref}
\usepackage{tikz}
\usepackage{multirow}
\usepackage{subcaption}
\usepackage{algorithm}
\usepackage{algorithmic}

\newcommand{\bm}[1]{\boldsymbol{#1}}
\newcommand{\bmr}[1]{\bm{\mr{#1}}}
\newcommand{\lj}{[ \hspace{-2pt} [}
\newcommand{\rj}{] \hspace{-2pt} ]}
\newcommand{\mb}[1]{\mathbb{#1}}
\newcommand{\mc}[1]{\mathcal{#1}}
\newcommand{\mr}[1]{\mathrm{#1}}
\newcommand{\jump}[1]{\lj #1 \rj}
\newcommand{\aver}[1]{ \{#1\}  }
\newcommand{\wt}[1]{ \widetilde{ #1}}
\newcommand{\tr}[1]{\ifmmode \mathrm{tr}\left( #1 \right) \else 
\text{tr} \left( #1 \right) \fi }
\newcommand{\vect}[2]{ \begin{bmatrix} #1 \\ #2 \\ \end{bmatrix}}
\newcommand{\vech}[3]{ \begin{bmatrix} #1 \\ #2 \\ #3 \\ \end{bmatrix}}
\newcommand\curl{\ifmmode \mathrm{curl} \else \text{curl}\fi}

\newcommand\MTh{\mc{T}_h}
\newcommand\MEh{\mc{E}_h}
\newcommand\un{\bm{\mr n}}
\renewcommand{\d}[1]{\mathrm d \boldsymbol{#1}}
\newcommand\pnorm[1]{\| #1 \|_{\bm{\mathrm{p}}}}
\newcommand\unorm[1]{\| #1 \|_{\bm{\mathrm{u}}}}

\newcommand\bu{\bm{\nu}}

\newcommand\comment[1]{}

\newtheorem{theorem}{Theorem}
\newtheorem{lemma}{Lemma}

\newtheorem{remark}{Remark}

\title[Sequential Least Squares Method]{A Sequential Least Squares
Method for Elliptic Equations in Non-Divergence Form}

\author[R. Li]{Ruo Li} \address{CAPT, LMAM and School of Mathematical
  Sciences, Peking University, Beijing 100871, P.R. China}
\email{rli@math.pku.edu.cn}

\author[F.-Y. Yang]{Fanyi Yang} \address{School of Mathematical
  Sciences, Peking University, Beijing 100871, P.R. China}
\email{yangfanyi@pku.edu.cn}

\begin{document}
\maketitle

% vim:spell:tw=70:fo+=Mn:cc=70
\begin{abstract}
  We develop a new least squares method for solving the second-order
  elliptic equations in non-divergence form. Two least-squares-type
  functionals are proposed for solving the equation in two sequential
  steps. We first obtain a numerical approximation to the gradient in
  a piecewise irrotational polynomial space. Then together with the
  numerical gradient, we seek a numerical solution of the primitive
  variable in the continuous Lagrange finite element space. The
  variational setting naturally provides an {\it a posteriori} error
  which can be used in an adaptive refinement algorithm. The error
  estimates under the $L^2$ norm and the energy norm for both two
  unknowns are derived. By a series of numerical experiments, we
  verify the convergence rates and show the efficiency of the adaptive
  algorithm.

  \noindent \textbf{keywords}: non-divergence form, least squares
  method, piecewise irrotational space, discontinuous Galerkin
  method. 
\end{abstract}

%%% Local Variables:
%%% mode: latex
%%% TeX-master: "nondivergence"
%%% End:

% vim:spell:tw=70:fo+=Mn:cc=70
\section{Introduction}
\label{sec:intro}
This work is concerned with the non-divergence form second-order
elliptic equation, which is often encountered in many applications
from areas such as probability and stochastic processes
\cite{Smears2013discontinuous}. In addition, such problems also
naturally arise as the linearization to fully nonlinear PDEs, as
obtained by applying the Newton's iterative method, see
\cite{Caffarelli1997properties, Feng2009mixed}. Due to the
non-divergence structure, it is invalid to derive a variational
formulation by applying the integration by parts. Instead, the
existence and uniqueness of the solutions to this problem are sought
in the strong sense, we refer to \cite{Feng2017finite,
Smears2013discontinuous, Smears2014discontinuous, Babuska1994special,
Gallistl2017variational} and the references therein for the
well-posedness of the solutions to the non-divergence form
second-order elliptic equation. 

Recently several finite element methods have been proposed, though
such a problem does not naturally fit within the standard Galerkin
framework. Conforming finite element methods require $H^2$-regularity
for approximating the strong solution, which naturally leads to a
$C^1$ finite element space \cite{Bohmer2008finite,
Bialecki1998convergence}. But the $C^1$ finite elements are sometimes
considered impractical. In \cite{Lakkis2011finite}, the authors
introduced a mixed finite element method with $C^0$ finite element
space via a finite element Hessian obtained in the same approximation
space. In \cite{Feng2017finite}, the authors proposed and analyzed a
finite element method with $C^0$ space by introducing an interior
penalty term. But the coefficient matrix is assumed to be continuous.
Gallistl introduced a conforming mixed finite element method based on
a least squares functional, we refer to \cite{Gallistl2017variational}
for more details. In \cite{Neilan2019discrete}, the authors proposed a
simple and convergent finite element method with $C^0$ finite element
space.  Based on discontinuous approximations, Smears and S\"uli
proposed a discontinuous Galerkin method where the optimal convergence
rate in $h$ with respect to broken $H^2$ norm is proven and the
authors have extended this method to the Hamilton-Jacobi-Bellman
equations \cite{Smears2014discontinuous, Gallistl2019mixed}.  Besides,
Wang et al proposed a weak Galerkin method and we refer to
\cite{Wang2018primal} for details. 

In this paper, we propose a new least squares finite element method
for solving the non-divergence elliptic problem. We rewrite the
equation into an equivalent first-order system as a fundamental
requirement in modern least squares method \cite{Bochev1998review}.
We employ two different approximation spaces to solve the gradient and
the primitive variable sequentially, which is motivated from the idea
in \cite{li2019sequential}.  We first define a least squares
functional to seek a numerical approximation to the gradient in a
piecewise irrotational polynomial space. Then we obtain the
approximation to the primitive variable with the numerical gradient by
solving another least squares problem in the standard $C^0$ finite
element space. Our method avoids solving a saddle-point problem of
mixed formulation, and in contrast to \cite{Smears2013discontinuous,
Neilan2019discrete, Gallistl2019mixed} our method only involves the
first-order operator in each step. We prove the convergence rates for
both variables in $L^2$ norm and energy norm. The least squares
functional naturally serves as an {\it a posteriori} error estimate
and we introduce an adaptive algorithm for solving the problem of low
regularity. By carrying out a series of numerical experiments, we
verify the convergence orders in the error estimates and illustrate
the efficiency of the adaptive algorithm. 

The rest of this paper is organized as follows. Section
\ref{sec:preliminaries} gives the notations that will be used
throughout the paper and defines the considered problem. In Section
\ref{sec:dgspace}, we introduce the piecewise irrotational
approximation space and give some basic properties of this space. In
Section \ref{sec:dlsm}, we propose the least squares method for both
two variables respectively and the error estimates are derived. In
Section \ref{sec:numericalresults}, a series of numerical experiments
are presented for testing the accuracy of the proposed scheme.

%%% Local Variables:
%%% mode: latex
%%% TeX-master: "nondivergence"
%%% End:

% vim:spell:tw=70:fo+=Mn:cc=70
\section{Preliminaries} 
\label{sec:preliminaries}
Let $\Omega \subset \mb{R}^d(d = 2, 3)$ be a bounded convex domain
with the boundary $\partial \Omega$. We denote by $\MTh$ a regular and
shape-regular subdivision of $\Omega$ into simplexes. Let $\MEh^i$ be
the set of all interior faces associated with the subdivision $\MTh$,
$\MEh^b$ the set of all faces lying on $\partial \Omega$ and then
$\MEh = \MEh^i \cup \MEh^b$. We define 
\begin{displaymath}
  h_K = \text{diam}(K), \quad \forall K \in \MTh, \quad h_e =
  \text{diam}(e), \quad \forall e \in \MEh,
\end{displaymath}
and we set $h = h_{\max} = \max_{K \in \MTh} h_K$.

We then introduce the trace operators commonly used in the DG
framework.  Let $K^+$ and $K^-$ be two adjacent elements sharing an
interior face $e = \partial K^+ \cap \partial K^- \subset \MEh^i$ with
the unit outward normal vectors $\un^+$ and $\un^-$, respectively. Let
$v$ and $\bm{q}$ be scalar-valued functions and vector-valued
functions that may be discontinuous across $\MEh^i$.  For $v^+ :=
v|_{e \subset \partial K^+}$, $v^- := v|_{e \subset \partial K^-}$,
$\bm{q}^+ := \bm{q}|_{e \subset \partial K^+}$, $\bm{q}^- :=
\bm{q}|_{e \subset \partial K^-}$, we set the average operator $\aver{
\cdot}$ as 
\begin{displaymath}
  \aver{v} := \frac{1}{2} \left( v^+ + v^- \right), \quad \aver{
  \bm{q}} := \frac{1}{2} \left( \bm{q}^+ + \bm{q}^- \right),
\end{displaymath}
and we set the jump operator $\jump{ \cdot}$ as 
\begin{displaymath}
  \begin{aligned}
    \jump{v} & := v^+\un^+ + v^- \un^-, \quad \jump{\bm{q} \cdot \un }
    := \bm{q}^+ \cdot \un^+ + \bm{q}^- \cdot \un^-, \\ 
    \jump{\bm{q} \times \un} &:= \bm{q}^+ \times \un^+ + \bm{q}^-
    \times \un^-, \quad   \jump{\bm{q} \otimes \un} := \bm{q}^+
    \otimes \un^+ + \bm{q}^- \otimes \un^-, \\
  \end{aligned}
\end{displaymath}
where $\cdot \otimes \cdot$ denotes the tensor product between two
vectors.  For $e \in \MEh^b$, these definitions shall be modified as
follows:
\begin{displaymath}
  \begin{aligned}
    \aver{v} &:= v, \quad \aver{\bm{q}} := \bm{q}, \quad \jump{v} :=
    v\un, \\
    \jump{\bm{q} \cdot \un }  &:= \bm{q} \cdot \un, \quad \jump{\bm{q}
    \times \un} := \bm{q} \times \un, \quad \jump{\bm{q} \otimes \un}
    := \bm{q} \otimes \un. \\
  \end{aligned}
\end{displaymath}

Throughout this paper, let us note that $C$ and $C$ with a subscript
are generic constants that may be different from line to line but are
independent of $h$. We will also use the standard notations and
definitions for the spaces $L^r(D)$, $L^r(D)^d$, $L^r(D)^{d \times d}$,
$H^r(D)$, $H^r(D)^d$, $H^r(D)^{d \times d}$ with $D$ a bounded domain
and $r$ a positive integer (may be $\infty$), and their associated
inner products and norms. We define the Sobolev space of irrotational
vector fields by
\begin{displaymath}
  \bmr{I}^r(D) := \left\{ \bm{v} \in H^r(D)^d \ |\ \nabla \times
  \bm{v} = 0 \quad \text{in } \Omega \right\}.
\end{displaymath}
Further, for the partition $\MTh$ we will follow the standard
definitions for the broken Sobolev spaces $L^2(\MTh)$, $L^2(\MTh)^d$,
$L^2(\MTh)^{d \times d}$, $H^r(\MTh)$, $H^r(\MTh)^d$, $H^r(\MTh)^{d
\times d}$ and their corresponding broken norms
\cite{arnold2002unified}. 

The problem dealt with in this paper is to find numerical
approximation to the strong solution for the elliptic problem in
non-divergence form, which reads
\begin{equation}
  \begin{aligned}
    \mc{L} u := A: D^2u  = f & \quad \text{in } \Omega, \\
    u = g & \quad \text{on } \partial \Omega, \\
  \end{aligned}
  \label{eq:nondiv}
\end{equation}
where $\cdot : \cdot$ denotes the Frobenious inner product between two
matrices. The coefficient matrix $A(\bm x) = \left\{ a_{ij}( \bm{x})
\right\} \in L^\infty(\Omega)^{d \times d}$ is assumed to be uniformly
elliptic, i.e. there exist two positive constants $\underline{\nu}$
and $\overline{\nu}$ satisfying 
\begin{displaymath}
  \underline{\nu} | \bm{\xi}|^2 \leq \bm{\xi}^T A(\bm{x}) \bm{\xi}
  \leq \overline{\nu} | \bm{\xi} |^2, \quad \forall \bm{\xi} \in
  \mb{R}^d, \quad \text{a.e. in } \Omega.
\end{displaymath}
We furthermore assume that the coefficient satisfies the Cordes
condition: there exists a positive constant $\varepsilon \leq 1$ such
that 
\begin{equation}
  \frac{|A|^2}{(\tr{A})^2} \leq \frac{1}{d - 1 + \varepsilon}, \quad
  \text{a.e. in } \Omega,
  \label{eq:cordes}
\end{equation}
where $|A| := \sqrt{ A:A}$ denotes the Frobenious norm. The uniform
ellipticity of the coefficient cannot ensure the well-posedness of the
problem \eqref{eq:nondiv}, at least in three dimensions. If the
condition \eqref{eq:cordes} holds, there exists a unique strong
solution $u \in H^2(\Omega)$ to \eqref{eq:nondiv} with the proper
source term $f$ and the boundary condition $g$, we refer to
\cite{Smears2013discontinuous, Feng2017finite,
Smears2014discontinuous} for more regularity results of the problem
\eqref{eq:nondiv}.  Particularly, the uniformly elliptic coefficient
$A$ directly implies the Cordes condition \eqref{eq:cordes} for the
planar case \cite{Smears2014discontinuous}. 

In this paper, we introduce the gradient variable $\bm{p} = \nabla u$
and the scalar elliptic problem \eqref{eq:nondiv} will be rewritten
into the first-order system:
\begin{equation}
  \begin{aligned}
    A: \nabla \bm{p} &= f \quad \text{in } \Omega, \\
    \bm{p} - \nabla u &= 0 \quad \text{in } \Omega, \\
    u &= g \quad \text{on } \partial \Omega. \\
  \end{aligned}
  \label{eq:upnondiv}
\end{equation}
To transform the problem into first-order system is one of the
fundamental ideas in modern least squares finite element method
\cite{Bochev1998review} and our proposed least squares method is based
on the formulation \eqref{eq:upnondiv}.

%%% Local Variables:
%%% mode: latex
%%% TeX-master: "nondivergence"
%%% End:

% vim:spell:tw=70:fo+=Mn:cc=70
\section{The finite element space}
\label{sec:dgspace}
In this section, we introduce the locally curl-free finite element
space $\bmr{S}_h^m$ with an integer $m \geq 0$, which is defined as
\begin{displaymath}
  \bmr{S}_h^m := \left\{ \bm{v} \in L^2(\Omega)^d\ |\ \bm{v}|_K \in
  \mb{P}_m(K)^d, \quad \nabla \times(\bm{v}|_K) = 0, \quad \forall K
  \in \MTh \right\}.
\end{displaymath}
We first give some basic properties of $\bmr{S}_h^m$ which are very
essential in the convergence analysis. We set $ \bmr{S}^m(D) :=
\mb{P}_m(D)^d \cap \bmr{I}^0(D)$ as the space of irrotational
polynomials of degree at most $m$ on the domain $D$. Obviously, we can
compactly write the space $\bmr{S}_h^m$ as $\bmr{S}_h^m = \Pi_{K \in
\MTh} \bmr{S}^m(K)$.
\begin{lemma}
  For any $\bm{q} \in \bmr{I}^{m + 1}(K) $ and an element $K \in
  \MTh$, there exists a polynomial $\wt{\bm{q}} \in \bmr{S}^m(K)$ such
  that
  \begin{equation}
    \begin{aligned} 
      \| \bm{q} - \wt{\bm{q}} \|_{H^k(K)} & \leq C h_K^{m + 1 - k} \|
      \bm{q} \|_{H^{m+1}(K)}, \quad 0 \leq k \leq m + 1. \\
%      \| \partial^k ( \bm{q} - \wt{\bm{q}} ) \|_{L^2(\partial K)} &
      %\leq C h_K^{m + d/2 - k} \| \bm{q} \|_{H^{m+1}(K)}, \quad 0
      %\leq k \leq m. \\
    \end{aligned}
    \label{eq:interpolationerror}
  \end{equation}
  \label{le:interpolationerror}
\end{lemma}
\begin{proof}
  Based on the fact that $\bmr{I}^{m + 1}(K) = \nabla H^{m+2}(K)$
  \cite{girault1986finite}, we have that there exists a function $v
  \in H^{m+2}(K) $ satisfying $\bm{q} = \nabla v$. We denote by
  $\wt{v} \in \mb{P}_{m+1}(K)$ the standard nodal interpolation
  polynomial of $v$. The estimate \eqref{eq:interpolationerror}
  is implied by the approximation property of $\wt{v}$ with
  $\wt{\bm{q}} = \nabla \wt{v} \in \bmr{S}^m(K)$, which completes the
  proof.
\end{proof}
For any $\bm{q} \in \bmr{I}^{m+1}(\Omega)$ and any element $K \in
\MTh$, we define a local $L^2$-projection $\pi_{K}^{\bmr{S}, m}$ such
that $\pi_{K}^{\bmr{S}, m} \bm{q} \in \bmr{S}^m(K)$ satisfies
\begin{equation}
  \|\bm{q} - \pi_{K}^{\bmr{S}, m} \bm{q} \|_{L^2(K)} = \min_{\bm{r}
  \in \bmr{S}^m(K)} \|\bm{q} - \bm{r} \|_{L^2(K)}. 
  \label{eq:localL2pro}
\end{equation}
Then the we can obtain the following local approximation property of
$\pi_K^{\bmr{S}, m}$ from Lemma \ref{le:interpolationerror}.
\begin{lemma}
  For any element $K \in \MTh$, the following estimates hold:
  \begin{equation}
    \begin{aligned}
      \|\bm{q} - \pi_{K}^{\bmr{S}, m} \bm{q} \|_{H^k(K)} & \leq C
      h_K^{m+1 - k} \|\bm{q} \|_{H^{m+1}(K)}, \quad 0 \leq k \leq m
      + 1, \\
      \| \partial^k ( \bm{q} -  \pi_{K}^{\bmr{S}, m} \bm{q}  )
      \|_{L^2(\partial K)} & \leq C h_K^{m + 1/2 - k} \| \bm{q}
      \|_{H^{m+1}(K)}, \quad  0 \leq k \leq m, \\
    \end{aligned}
    \label{eq:L2interpolation}
  \end{equation}
  for any $\bm{q} \in \bmr{I}^{m+1}(\Omega)$.
  \label{le:L2interpolation}
\end{lemma}
\begin{proof}
  Obviously from \eqref{eq:localL2pro} one has that 
  \begin{displaymath}
    \pi_{K}^{\bmr{S}, m} \bm{r} = \bm{r}, \quad \forall \bm{r} \in
    \bmr{S}^m(K).
  \end{displaymath}
  Applying the inverse inequality directly leads to 
  \begin{displaymath}
    \begin{aligned}
      \|\bm{q} - \pi_{K}^{\bmr{S}, m} \bm{q} \|_{H^k(K)} &\leq \|\bm{q}
      - \wt{\bm{q}} \|_{H^k(K)}  + \| \pi_{K}^{\bmr{S}, m}(
      \wt{\bm{q}} -  \bm{q}) \|_{H^k(K)} \\
      & \leq  \|\bm{q} - \wt{\bm{q}} \|_{H^k(K)}  + Ch_K^{-k}  \|
      \pi_{K}^{\bmr{S}, m}( \wt{\bm{q}} -  \bm{q}) \|_{L^2(K)} \\
      & \leq  \|\bm{q} - \wt{\bm{q}} \|_{H^k(K)}  + Ch_K^{-k}  \|
      \wt{\bm{q}} -  \bm{q} \|_{L^2(K)} + Ch_K^{-k} \| \bm{q} -
      \pi_{K}^{\bmr{S}, m} \bm{q} \|_{L^2(K)} \\ 
      & \leq Ch_K^{m+1 - k} \|\bm{q} \|_{H^{m+1}(K)}, \\
    \end{aligned}
  \end{displaymath}
  where $\wt{\bm{q}}$ is defined in Lemma \ref{le:interpolationerror}.
  Similarly, by the trace inequality it is trivial to obtain the trace
  estimate in \eqref{eq:L2interpolation}, which completes the proof.
\end{proof}
Furthermore, we define a global $L^2$-projection $\Pi_h^{\bmr{S}, m}$
in a piecewise manner: for any $\bm{q} \in \bmr{I}^{m+1}(\Omega)$, $
\Pi_h^{\bmr{S}, m} \bm{q} \in \bmr{S}_h^m$ is denoted by
\begin{displaymath}
  (\Pi_h^{\bmr{S}, m} \bm{q})|_K = \pi_K^{\bmr{S}, m} \bm{q}, \quad
  \forall K \in \MTh.
\end{displaymath}
Clearly, the global $L^2$-projection has the following approximation
property:
\begin{lemma}
  For any element $K \in \MTh$, the following estimates hold:
  \begin{equation}
    \begin{aligned}
      \|\bm{q} - \Pi_{h}^{\bmr{S}, m} \bm{q} \|_{H^k(K)} & \leq C
      h_K^{m+1 - k} \|\bm{q} \|_{H^{m+1}(K)}, \quad 0 \leq k \leq m
      + 1, \\
      \| \partial^k ( \bm{q} -  \Pi_{h}^{\bmr{S}, m} \bm{q}  )
      \|_{L^2(\partial K)} & \leq C h_K^{m + 1/2 - k} \| \bm{q}
      \|_{H^{m+1}(K)}, \quad  0 \leq k \leq m, \\
    \end{aligned}
    \label{eq:globalL2interpolation}
  \end{equation}
  for any $\bm{q} \in \bmr{I}^{m+1}(\Omega)$.
  \label{le:globalL2interpolation}
\end{lemma}
\begin{proof}
  It is a direct extension of Lemma \ref{le:L2interpolation}.
\end{proof}

We define $\bm{V}_h^m$ and $V_h^m$ as the piecewise polynomial spaces, 
\begin{displaymath}
  V_h^m := \left\{ v_h \in L^2(\Omega) \  | \  v_h|_K \in \mb{P}_m(K),
  \quad \forall K \in \MTh \right\}, \quad \bm{V}_h^m := (V_h^m)^d.
\end{displaymath}
For the analysis of convergence, we will require the following
estimates.

\begin{lemma}
  The following estimates holds, 
  \begin{equation}
    \| \nabla \bm{p}_h \|_{L^2(\Omega)} \leq \|\nabla \cdot \bm{p}_h
    \|_{L^2(\Omega)} + \|\nabla \times \bm{p}_h \|_{L^2(\Omega)} + C
    \left( \sum_{e \in \MEh^b} \frac{1}{h_e} \| \bm{p}_h \times \un
    \|_{L^2(e)}^2 \right)^{1/2},
    \label{eq:dMaxwell}
  \end{equation}
  for any $\bm{p}_h \in \bm{V}_h^m \cap H^1(\Omega)^d$.
  \label{le:dMaxwell}
\end{lemma}

\begin{proof}
  Actually $\bm{V}_h^m \cap H^1(\Omega)$ is the vector-valued Lagrange
  finite element space of degree $m$. We let $\mc{N} = \left\{ \nu_0,
  \nu_1, \ldots, \nu_n\right\}$ denote the Lagrange points 
  corresponding to the triangular (tetrahedral) partition $\MTh$, and
  we let $\left\{ \phi_{\bm{\nu}_0}, \phi_{\bm{\nu}_1}, \ldots,
  \phi_{\bm{\nu}_n} \right\}$ denote the corresponding Lagrange basis
  functions, which satisfy $\phi_{\bm{\nu}_i}(\bm{\nu}_j) =
  \delta_{ij}$. Hence, there exists a group of coefficients $\left\{
  \alpha_{\bu}^j \right\}(1 \leq j \leq d, \bu \in \mc{N})$ which
  allows us to write $\bm{p}_h = (p_h^1, \ldots, p_h^d)^T$ as 
  \begin{displaymath}
    \bm{p}_h = \begin{bmatrix}
      p_h^1 \\ \ldots \\ p_h^d 
    \end{bmatrix} = \begin{bmatrix}
      \sum_{\bu \in \mc{N}} \alpha_{\bu}^1 \phi_{\bu} \\ \ldots \\
      \sum_{\bu \in \mc{N}} \alpha_{\bu}^d \phi_{\bu} 
    \end{bmatrix} = \sum_{\bu \in \mc{N}} \bm{\alpha}_{\bu}
    \phi_{\bu},
  \end{displaymath}
  where $\bm{\alpha}_{\bu} = (\alpha_{\bu}^1, \ldots,
  \alpha_{\bu}^d)^T$. Then we divide the points in $\mc{N}$ into three
  categories, 
  \begin{equation}
    \begin{aligned}
      \mc{N}_i &:= \left\{ \bm{\nu} \in \mc{N} \ | \  \bm{\nu} \text{
      is interior to the domain } \Omega \right\}, \\
      \mc{N}_v &:= \hspace{-3.6pt} \begin{aligned}
        & \left\{ \bm{\nu} \in \mc{N} \ | \  \bm{\nu} \text{ is a
        vertex of the polygonal boundary} \partial \Omega \right\},
        \quad d = 2, \\
        & \left\{ \bm{\nu} \in \mc{N} \ | \  \bm{\nu} \text{ lies on
        an edge of the polyhedral boundary } \partial \Omega
        \right\}, \quad d = 3, \\
      \end{aligned} \\
      \mc{N}_b &:= \mc{N} \backslash ( \mc{N}_i \cup \mc{N}_v). \\
    \end{aligned}
    \label{eq:Ndef}
  \end{equation}
  From the definition \eqref{eq:Ndef}, we note that for any point $\bu
  \in \mc{N}_b$, there exists a face $e_{\bu} \in \MEh^b$ such that
  $\bu \in e$, and for any point $\bu \in \mc{N}_v$, there exist two
  nonparallel faces $e_{\bu}^1, e_{\bu}^2 \in \MEh^b$ such that $\bu
  \in e_{\bu}^1 \cap e_{\bu}^2$.

  Then we construct a new group of  coefficients $ \{
  \beta_{\bm{\nu}}^j \} (1 \leq j \leq d, \nu \in \mc{N})$ such that 
  \begin{equation}
    \beta_{\bm{\nu}}^j = \begin{cases}
      \alpha_{\bm{\nu}}^j, & \bm{\nu} \in \mc{N}_i, \\
      \wt{\beta}_{\bm{\nu}}^j, & \bm{\nu} \in \mc{N}_b, \\
      0, & \bm{\nu} \in \mc{N}_v, \\
    \end{cases}
    \quad 1 \leq j \leq d.
    \label{eq:newbeta}
  \end{equation}
  For any point $\bu \in \mc{N}$, we also let $\bm{\beta}_{\bu} =
  (\beta_{\bu}^1, \ldots, \beta_{\bu}^d)^T$. For any point $\bu \in
  \mc{N}_b$, we determine $\wt{\beta}_{\bu}^j$ by the following
  equations, 
  \begin{equation}
    {\bm{\beta}}_{\bm{\nu}} \times \un = \bm{0}, \quad
    {\bm{\beta}}_{\bm{\nu}} \cdot \un = \bm{\alpha}_{\bm{\nu}} \cdot
    \un,
    \label{eq:tildebeta}
  \end{equation}
  where $\un$ denotes the unit outward normal corresponding to
  $e_{\bu}$. Then we define a new polynomial ${\bm{q}}_h$ as 
  \begin{displaymath}
    \bm{q}_h = \begin{bmatrix}
      q_h^1 \\ \ldots \\ q_h^d 
    \end{bmatrix} = \begin{bmatrix}
      \sum_{\bu \in \mc{N}} \beta_{\bu}^1 \phi_{\bu} \\ \ldots \\
      \sum_{\bu \in \mc{N}} \beta_{\bu}^d \phi_{\bu} 
    \end{bmatrix} = \sum_{\bu \in \mc{N}} \bm{\beta}_{\bu}
    \phi_{\bu}.
  \end{displaymath}
  Then we will estimate the error $ \|\nabla(\bm{p}_h - \bm{q}_h
  )\|_{L^2(\Omega)}^2$. From \eqref{eq:newbeta}, we have that 
  \begin{displaymath}
     \|\nabla(\bm{p}_h - {\bm{q}}_h )\|_{L^2(\Omega)}^2 \leq
     C \sum_{\nu \in \mc{N}_b \cup \mc{N}_v} \| \nabla \phi_{\bm{\nu}}
     \|_{L^2(\Omega)}^2 | \bm{\alpha}_{\bm{\nu}} -
     \bm{\beta}_{\bm{\nu}} |^2.
  \end{displaymath}
  For any point $\bu \in \mc{N}_b$, the scaling argument
  \cite{Karakashian2007convergence} gives that $\|\nabla
  \phi_{\bm{\nu}} \|_{L^2(\Omega)}^2 \leq C h_{e_{\bu}}^{d -
  2}$. Then we deduce that 
  \begin{displaymath}
    \begin{aligned}
      \sum_{\bm{\nu} \in \mc{N}_b} \| \nabla & \phi_{\bm{\nu}}
      \|_{L^2(\Omega)}^2   | \bm{\alpha}_{\bm{\nu}} -
      \bm{\beta}_{\bm{\nu}} |^2 \leq C  \sum_{\bm{\nu} \in \mc{N}_b}
      h_{e_{\bu}}^{d - 2} \left( | ( \bm{\alpha}_{\bm{\nu}} -
      \bm{\beta}_{\bm{\nu}}) \times \un |^2 + | (
      \bm{\alpha}_{\bm{\nu}} - \bm{\beta}_{\bm{\nu}}) \cdot \un |^2
      \right) \\
      & = C  \sum_{\bm{\nu} \in \mc{N}_b} h_{e_{\bu}}^{d - 2}  |
      \bm{\alpha}_{\bm{\nu}} \times \un |^2 = C  \sum_{\bm{\nu} \in
      \mc{N}_b} h_{e_{\bu}}^{d - 2}  | \bm{p}_h(\bm{\nu}) \times \un
      |^2  \\
      & \leq C  \sum_{\bm{\nu} \in \mc{N}_b} h_{e_{\bu}}^{d - 2} \|
      \bm{p}_h \times \un \|_{L^\infty(e_{\bu})}^2 \leq C
      \sum_{\bm{\nu} \in \mc{N}_b} h_{e_{\bu}}^{-1 } \| \bm{p}_h
      \times \un \|_{L^2(e_{\bu})}^2 \leq  C \sum_{e \in \MEh^b}
      h_e^{-1 } \| \bm{p}_h \times \un \|_{L^2(e)}^2,
    \end{aligned}
  \end{displaymath}
  where $\un$ denotes the unit outward normal corresponding to
  $e_{\bu}$. For any point $\bm{\nu} \in \mc{N}_v$, there exist two
  faces $e_{\bu}^1, e_{\bu}^2 \in \MEh^b$ such that $\bu \in e_{\bu}^1
  \cap e_{\bu}^2$, and we denote $\un_1, \un_2$ as their corresponding
  unit outward normal. Since $\un_1$ and $\un_2$ are not parallel, we
  have that there exists a constant $C$ that only depends on $\Omega$
  such that 
  \begin{displaymath}
    |\bm{v}|^2 \leq C \left( |\bm{v} \times \un_1|^2 + |\bm{v} \times
    \un_2|^2 \right), 
  \end{displaymath}
  for any $\bm{v} \in \mathbb{R}^d$. By the inverse inequality, we
  derive that 
  \begin{displaymath}
    \begin{aligned}
      \sum_{\bm{\nu} \in \mc{N}_v} \| \nabla \phi_{\bm{\nu}}
      & \|_{L^2(\Omega)}^2  | \bm{\alpha}_{\bm{\nu}} -
      \bm{\beta}_{\bm{\nu}} |^2 =  \sum_{\bm{\nu} \in \mc{N}_v} \|
      \nabla \phi_{\bm{\nu}} \|_{L^2(\Omega)}^2 |
      \bm{\alpha}_{\bm{\nu}}|^2 \leq C  \sum_{\bm{\nu} \in \mc{N}_v}
      \| \nabla \phi_{\bm{\nu}} \|_{L^2(\Omega)}^2   \left(
      |\bm{\alpha}_{\bm{\nu}} \times \un_1|^2 +
      |\bm{\alpha}_{\bm{\nu}} \times \un_2|^2 \right) \\ 
      & \leq C \sum_{\bm{\nu} \in \mc{N}_v} \left( h_{e_{\bu}^1}^{d -
      2} |\bm{\alpha}_{\bm{\nu}} \times \un_1|^2 + h_{e_{\bu}^2}^{d -
      2} |\bm{\alpha}_{\bm{\nu}} \times \un_2|^2 \right) \\
      & \leq C \sum_{\bm{\nu} \in \mc{N}_v} \left( h_{e_{\bu}^1}^{d -
      2} | \bm{p}_h(\bm{\nu}) \times \un_1|^2 + h_{e_{\bu}^2}^{d - 2}
      |\bm{p}_h(\bm{\nu}) \times \un_2|^2 \right) \\
      & \leq C \sum_{\bm{\nu} \in \mc{N}_v} \left( h_{e_{\bu}^1}^{d -
      2} \| \bm{p}_h \times \un_1 \|_{L^\infty(e_{\bu}^1)}^2 +
      h_{e_{\bu}^2}^{d - 2} \|\bm{p}_h \times \un_2
      \|_{L^\infty(e_{\bu}^2)}^2 \right) \\
      & \leq C \sum_{\bm{\nu} \in \mc{N}_v} \left( h_{e_{\bu}^1}^{ -1}
      \| \bm{p}_h \times \un_1 \|_{L^2 (e_{\bu}^1)}^2 +
      h_{e_{\bu}^2}^{-1} \|\bm{p}_h \times \un_2 \|_{L^2(e_{\bu}^2)}^2
      \right) \leq  C \sum_{e \in \MEh^b} h_e^{-1 } \| \bm{p}_h \times
      \un \|_{L^2(e)}^2.   \\
    \end{aligned}
  \end{displaymath}
  Collecting all estimates above, we arrive at the estimate 
  \begin{equation}
    \| \nabla (\bm{p}_h - \bm{q}_h) \|_{L^2(\Omega)}^2 \leq C
    \sum_{e \in \MEh^b} h_e^{-1} \| \bm{p}_h \times \un \|_{L^2(e)}^2.
    \label{eq:qtqdiff}
  \end{equation}
  Further, it is trivial to check the tangential trace $\bm{q}_h
  \times \un$ vanishes on $\partial \Omega$. Applying the Maxwell
  inequality \cite{girault1986finite}, we get that 
  \begin{equation}
    \|\nabla {\bm{q}}_h \|_{L^2(\Omega)} \leq \|\nabla \cdot
    {\bm{q}}_h \|_{L^2(\Omega)} +  \|\nabla \times {\bm{q}}_h
    \|_{L^2(\Omega)}. 
    \label{eq:maxwellestimate}
  \end{equation}
  By \eqref{eq:qtqdiff}, \eqref{eq:maxwellestimate} and triangle
  inequality, we obtain that 
  \begin{displaymath}
    \begin{aligned}
      \|\nabla \bm{p}_h \|_{L^2(\Omega)}  &\leq \|\nabla {\bm{q}}_h
      \|_{L^2(\Omega)} + \| \nabla (\bm{p}_h - {\bm{q}}_h)
      \|_{L^2(\Omega)} \\ 
      & \leq  \|\nabla \cdot {\bm{q}}_h \|_{L^2(\Omega)} + \|\nabla
      \times {\bm{q}}_h \|_{L^2(\Omega)}  + \| \nabla (\bm{p}_h -
      {\bm{q}}_h) \|_{L^2(\Omega)}  \\
      & \leq \|\nabla \cdot \bm{p}_h \|_{L^2(\Omega)}  + \|\nabla
      \times \bm{p}_h \|_{L^2(\Omega)}  \\ 
      & + \| \nabla \cdot (\bm{p}_h -
      {\bm{q}}_h) \|_{L^2(\Omega)}  + \| \nabla \times (\bm{p}_h -
      {\bm{q}}_h) \|_{L^2(\Omega)}  + \| \nabla (\bm{p}_h -
      {\bm{q}}_h) \|_{L^2(\Omega)} \\
      & \leq \|\nabla \cdot \bm{p}_h \|_{L^2(\Omega)} +  \|\nabla
      \times \bm{p}_h \|_{L^2(\Omega)} + C  \| \nabla
      (\bm{p}_h - {\bm{q}}_h) \|_{L^2(\Omega)} \\
      & \leq \|\nabla \cdot \bm{p}_h \|_{L^2(\Omega)} +  \|\nabla
      \times \bm{p}_h \|_{L^2(\Omega)} +C  \left( \sum_{e \in \MEh^b}
      h_e^{-1} \| \bm{p}_h \times \un \|_{L^2(e)}^2 \right)^{1/2},
    \end{aligned}
  \end{displaymath}
  which gives us the estimate \eqref{eq:dMaxwell} and completes the
  proof.
\end{proof}

\begin{lemma}
  The following estimate holds,
  \begin{equation}
    \| \nabla \bm{p}_h \|_{L^2(\MTh)} \leq \| \nabla \cdot \bm{p}_h
    \|_{L^2(\MTh)} + C \left( \sum_{e \in \MEh^i} \frac{1}{h_e} \|
    \jump{\bm{p}_h \otimes \un} \|^2_{L^2(e)} + \sum_{e \in \MEh^b}
    \frac{1}{h_e} \|  \bm{p}_h \times \un \|_{L^2(e)}^2 \right)^{1/2},
    \label{eq:discreteMT}
  \end{equation}
  for any $\bm{p}_h \in \bmr{S}_h^m$.
  \label{le:discreteMT}
\end{lemma}

\begin{proof}
  Clearly, $\bm{p}_h \in \bm{V}_h^m$.  By \cite[Theorem
  2.1]{Karakashian2007convergence}, there exists a polynomial
  ${\bm{q}}_h \in \bm{V}_h^m \cap H^1(\Omega)^d$ satisfying the
  estimate, 
  \begin{equation}
    \|\nabla^\alpha (\bm{p}_h - {\bm{q}}_h) \|_{L^2(\MTh)}^2 \leq C
    \sum_{e \in \MEh^i} h_e^{1-2\alpha} \| \jump{\bm{p}_h \otimes \un}
    \|^2_{L^2(e)}, \quad \alpha = 0, 1.
    \label{eq:pqdiff}
  \end{equation}
  Lemma \ref{le:dMaxwell} and the estimate \eqref{eq:dMaxwell} imply
  that 
  \begin{displaymath}
    \| \nabla \bm{q}_h \|_{L^2(\Omega)} \leq \|\nabla \cdot \bm{q}_h
    \|_{L^2(\Omega)} + \|\nabla \times \bm{q}_h \|_{L^2(\Omega)} + C
    \left( \sum_{e \in \MEh^b} \frac{1}{h_e} \| \bm{q}_h \times \un
    \|_{L^2(e)}^2 \right)^{1/2},
  \end{displaymath}
  Hence, we have that
  \begin{displaymath}
    \begin{aligned}
      \|\nabla \bm{p}_h \|_{L^2(\MTh)} & \leq \| \nabla \bm{q}_h
      \|_{L^2(\Omega)} + \| \nabla (\bm{p}_h - \bm{q}_h)
      \|_{L^2(\MTh)} \\
      & \leq \|\nabla \cdot \bm{q}_h \|_{L^2(\Omega)} + \|\nabla
      \times \bm{q}_h \|_{L^2(\Omega)} \\ 
      & + C \left( \sum_{e \in \MEh^b} \frac{1}{h_e} \| \bm{q}_h
      \times \un \|_{L^2(e)}^2 \right)^{1/2} + \|\nabla (\bm{p}_h -
      \bm{q}_h) \|_{L^2(\MTh)} \\ 
      & \leq \|\nabla \cdot \bm{p}_h \|_{L^2(\MTh)} + \|\nabla
      \times \bm{p}_h \|_{L^2(\MTh)} \\
      & + C \left( \sum_{e \in \MEh^b} \frac{1}{h_e} \| \bm{q}_h
      \times \un \|_{L^2(e)}^2 \right)^{1/2} + C \|\nabla (\bm{p}_h -
      \bm{q}_h) \|_{L^2(\MTh)} \\ 
      & \leq \|\nabla \cdot \bm{p}_h \|_{L^2(\MTh)} + C \left( \sum_{e
      \in \MEh^b} \frac{1}{h_e} \| \bm{q}_h \times \un \|_{L^2(e)}^2 +
      \sum_{e \in \MEh^i} \frac{1}{h_e} \jump{ \bm{p}_h \otimes \un}
      \|_{L^2(e)}^2 \right)^{1/2}
    \end{aligned}
  \end{displaymath}
  By the trace estimate, we deduce that
  \begin{displaymath}
    \begin{aligned}
      \sum_{e \in \MEh^b} h_e^{-1} \| \bm{q}_h \times \un
      \|_{L^2(e)}^2 & \leq C \left( \sum_{e \in \MEh^b} h_e^{-1} \|
      \bm{p}_h \times \un \|_{L^2(e)}^2  +  \sum_{e \in \MEh^b}
      h_e^{-1} \| (\bm{p}_h - \bm{q}_h) \times \un \|_{L^2(e)}^2
      \right)  \\
      &\leq C \left( \sum_{e \in \MEh^b} h_e^{-1} \| \bm{p}_h \times \un
      \|_{L^2(e)}^2 + \sum_{K \in \MTh} \| \nabla(\bm{p}_h - \bm{q}_h)
      \|_{L^2(K)}^2 \right) \\
      & \leq C  \left( \sum_{e \in \MEh^b} h_e^{-1} \| \bm{p}_h \times
      \un \|_{L^2(e)}^2 + \sum_{e \in \MEh^i} h_e^{-1} \|
      \jump{\bm{p}_h \otimes \un} \|^2_{L^2(e)} \right). 
    \end{aligned}
  \end{displaymath}
  Combining the estimates above can yield the estimate
  \eqref{eq:discreteMT}, which completes the proof.
\end{proof}
To end this section, we outline a method for constructing bases for
the space $\bmr{S}_h^m$. One can take the gradient of the natural
basis polynomials 
\begin{displaymath}
  1, x, y, x^2, xy, y^2, \cdots
\end{displaymath}
to get a basis for the finite elements of $\bmr{S}_h^m$. For an
instance, in two dimensions if linear accuracy is considered, one
could obtain the basis functions,
\begin{displaymath}
  \vect{1}{0}, \ \vect{0}{1},\ \vect{x}{0}, \  \vect{0}{y}, \
  \vect{y}{x}. 
\end{displaymath}
Furthermore, there are also 4 second-order and 5 third-order basis
functions: 
\begin{displaymath}
  \vect{x^2}{0}, \ \vect{2xy}{x^2}, \ \vect{y^2}{2xy}, \
  \vect{0}{y^2}, 
\end{displaymath}
and
\begin{displaymath}
  \vect{x^3}{0}, \ \vect{3x^2y}{x^3}, \ \vect{xy^2}{x^2y}, \
  \vect{y^3}{3xy^2}, \ \vect{0}{y^3}.
\end{displaymath}
For the case $d = 3$, the basis functions could be constructed in a
similar way: for $m=1$, there are 9 basis functions which read
\begin{displaymath}
  \vech{1}{0}{0}, \ \vech{0}{1}{0}, \ \vech{0}{0}{1}, \ 
  \vech{x}{0}{0}, \ \vech{y}{x}{0}, \ \vech{z}{0}{x}, \ 
  \vech{0}{y}{0}, \ \vech{0}{z}{y}, \ \vech{0}{0}{z}.
\end{displaymath}
In our implementation, a normalization and a translation of the
coordinates is applied to guarantee the numerical stability
\cite{Liu2008note}. Taking 2D case as an example, we denote $(X, Y)$
in each element by
\begin{displaymath}
  X = \frac{x - x_c}{\sqrt{T}}, \quad Y = \frac{y - y_c}{\sqrt{T}},
\end{displaymath}
where $(x_c, y_c)$ is the barycenter of the triangular element and $T$
is its area. Substituting $(X, Y)$ for $(x, y)$ in these basis
functions could share a better numerical stability while the local
irrotational property still holds. 

%%% Local Variables:
%%% mode: latex
%%% TeX-master: "nondivergence"
%%% End:

% vim:spell:tw=70:fo+=Mn:cc=70
\section{Sequential Least Squares Method}
\label{sec:dlsm}
In this section, we consider a least squares method based on the
first-order system \eqref{eq:upnondiv} to approximate $\bm{p}$ and $u$
sequentially. Let us first define a least squares functional
$J_h^{\bm{\mr{p}}}(\cdot)$ by 
\begin{equation}
  \begin{aligned}
    J_h^{\bm{\mr{p}}}( \bm{q}) := \sum_{K \in \MTh} \| A:\nabla \bm{q}
    - f \|_{L^2(K)}^2 & + \sum_{e \in \MEh^i} \frac{\mu}{h_e} \|
    \jump{ \bm{q} \otimes \un } \|_{L^2(e)}^2 \\ 
    & + \sum_{e \in \MEh^b} \frac{\mu}{h_e} \| \bm{q} \times \un -
    \nabla g \times \un \|_{L^2(e)}^2,
  \end{aligned}
  \label{eq:pfunc}
\end{equation}
for seeking a numerical approximation of the variable $\bm{p}$.  The
functional $J_h^{\bm{\mr{p}}}( \cdot)$ consists of the part related to
the gradient $\bm{p}$ in \eqref{eq:upnondiv} and the terms on the
faces, and $\mu$ is the penalty parameter which will be specified
later on. We note that the boundary condition in \eqref{eq:upnondiv}
provides the tangential trace of the gradient on the boundary.
Minimizing the problem \eqref{eq:pfunc} in the space $\bmr{S}_h^m$
will give an approximation to the gradient $\bm{p}$, which reads
\begin{equation}
  \inf_{\bm{q}_h \in \bmr{S}_h^m} J_h^{\bm{\mr{p}}}(\bm{q}_h).
  \label{eq:infp}
\end{equation}
Thus, the corresponding variational equation takes the form: find
$\bm{p}_h \in \bmr{S}_h^m$ such that
\begin{equation}
  a^{\bm{\mr{p}}}_h(\bm{p}_h, \bm{q}_h) = l_h^{\bm{\mr{p}}}(\bm{q}_h),
  \quad \forall \bm{q}_h \in \bmr{S}_h^m,
  \label{eq:varp}
\end{equation}
where the bilinear form $a_h^{\bm{\mr{p}}}(\cdot, \cdot)$ is 
\begin{equation}
  \begin{aligned}
    a_h^{\bm{\mr{p}}}(\bm{p}_h, \bm{q}_h) = \sum_{K \in \MTh} \int_K
    (A: \nabla \bm{p}_h)(A: \nabla \bm{q}_h) \d{x} &+ \sum_{e \in
    \MEh^i} \int_e \frac{\mu}{h_e} \jump{\bm{p} \otimes \un}
    \jump{\bm{q}_h \otimes \un } \d{s} \\
    &+ \sum_{e \in \MEh^b} \int_e \frac{\mu}{h_e} (\bm{p}_h \times
    \un) \cdot (\bm{q}_h \times \un) \d{s},  \\
  \end{aligned}
  \label{eq:bilinear}
\end{equation}
and the linear form $l_h^{\bm{\mr{p}}}(\cdot) $ is 
\begin{displaymath}
  l_h^{\bm{\mr{p}}}(\bm{q}_h) = \sum_{K \in \MTh} \int_K f(A:\nabla
  \bm{q}_h) \d{x} + \sum_{e \in \MEh^b} \int_e \frac{\mu}{h_e}
  (\bm{q}_h \times \un) \cdot (\nabla g \times \un) \d{s}.
\end{displaymath}
We follow \cite{Smears2013discontinuous, Smears2014discontinuous} to
define a constant $\gamma$ as
\begin{equation}
  \gamma = \frac{\tr{A}}{|A|^2},
  \label{eq:gamma}
\end{equation}
and the Cordes condition \eqref{eq:cordes} provides the following
inequality. 
\begin{lemma}
  Let $\gamma$ be defined by \eqref{eq:gamma} and $A(\bm{x}) \in
  L^\infty(\Omega)^{d\times d}$ satisfy Cordes condition, then for
  any matrix $B \in \mb{R}^{d \times d}$ we have that
  \begin{equation}
    |\gamma A: B - \tr{B} | \leq \sqrt{1 - \varepsilon} |B|,
    \label{eq:gammaineq}
  \end{equation}
  where $\varepsilon$ is given in \eqref{eq:cordes}.
  \label{le:gammaineq}
\end{lemma}
\begin{proof}
  By direct calculation, we obtain 
  \begin{displaymath}
    \begin{aligned} 
      |\gamma A: B - \tr{B} | & = \left| \sum_{i, j=1}^d (\gamma
      a_{ij} - \delta_{ij})b_{ij} \right| \leq \left( \sum_{i,j=1}^d |
      \gamma a_{ij} - \delta_{ij} |\right)^{1/2} |B| \\
      & \leq \sqrt{1 - \varepsilon} |B|, \\
    \end{aligned}
  \end{displaymath}
  which completes the proof.
\end{proof}
In particular, for any $\bm{q}_h \in \bmr{S}_h^m$ we set $B = \nabla
\bm{q}_h$ in \eqref{eq:gammaineq} and one has the following
estimate:
\begin{equation}
  |\gamma A:\nabla \bm{q}_h - \nabla \cdot \bm{q}_h | \leq \sqrt{1 -
  \varepsilon} | \nabla \bm{q}_h|, \quad \text{a.e. in } \Omega,
  \label{eq:gammaAp}
\end{equation}
which is central in the convergence analysis. 

Further we will focus on the continuity and coercivity of the
bilinear form $a_h^{\bm{\mr{p}}}(\cdot, \cdot)$. We begin by
introducing an energy norm $\pnorm{\cdot}$:
\begin{displaymath}
  \pnorm{ \bm{q}} := \left( \sum_{K \in \MTh} \| \nabla \bm{q}
  \|_{L^2(K)}^2 + \sum_{e \in \MEh^i} \frac{1}{h_e} \| \jump{ \bm{q}
  \otimes \un} \|_{L^2(e)}^2 + \sum_{e \in \MEh^b} \frac{1}{h_e} \|
  \bm{q} \times \un \|_{L^2(e)}^2 \right)^{1/2},
\end{displaymath}
for any $\bm{q} \in H^1(\MTh)^d$. We present the following lemma to
give a lower bound for the energy norm $\pnorm{\cdot}$.
\begin{lemma}
  For any $\bm{q} \in H^1(\MTh)^d$, the following inequality holds:
  \begin{equation}
    \| \bm{q} \|_{H^1(\MTh)} \leq C \pnorm{\bm{q}}.
    \label{eq:pupperbound}
  \end{equation}
  \label{le:pupperbound}
\end{lemma}
\begin{proof}
  It is sufficient to prove $ \| \bm{q} \|_{L^2(\Omega)} \leq C
  \pnorm{\bm{q}}$ for the estimate \eqref{eq:pupperbound}. To do so,
  we apply the Helmholtz decomposition of $L^2(\Omega)^d$. Here we
  proof for the planar case and it is trivial to extend the proof in
  three dimensions. Since $\bm{q} \in L^2(\Omega)$, there exist
  functions $v \in H^1_0(\Omega)$ and $\phi \in H^1(\Omega)$ such that 
  \begin{displaymath}
    \bm{q} = \nabla v + \nabla^\perp \times \phi := \begin{bmatrix}
      \partial_x v \\ \partial_y v \\
    \end{bmatrix} + \begin{bmatrix}
      \partial_y \phi \\ - \partial_x \phi \\
    \end{bmatrix},
  \end{displaymath}
  and the following stability holds
  \begin{displaymath}
    \|v\|_{H^1(\Omega)} + \| \phi \|_{H^1(\Omega)} \leq C \| \bm{q}
    \|_{L^2(\Omega)}.
  \end{displaymath}
  We refer to \cite{girault1986finite, Bensow2005discontinuous} for
  the detail of the decomposition. Then applying the integration by
  parts, together with the Helmholtz decomposition, we deduce that 
  \begin{displaymath}
    \begin{aligned}
      &\|\bm{q} \|_{L^2(\Omega)}^2  = \sum_{K \in \MTh} \int_K \nabla v
      \cdot \bm{q} \d{x} + \sum_{K \in \MTh} \int_K (\nabla^\perp
      \times \phi) \cdot \bm{q} \d{x} \\
      & = -\sum_{K \in \MTh} \int_K v \nabla \cdot \bm{q} \d{x} +
      \sum_{e \in \MEh^i} \int_e \jump{\bm{q} \cdot \un} v \d{s} -
      \sum_{K \in\MTh} \int_K \phi \nabla \times \bm{q} \d{x} +
      \sum_{e \in \MEh} \int_e \jump{\bm{q} \times \un} \phi \d{s}.
    \end{aligned}
  \end{displaymath}
  For the first term and third term, using the Cauchy-Schwarz
  inequality and the regularity estimate implies 
  \begin{displaymath}
    \begin{aligned} 
      \sum_{K \in \MTh} \int_K v \nabla \cdot \bm{q} \d{x} +  \sum_{K
      \in\MTh} \int_K \phi \nabla \times & \bm{q} \d{x} \leq C \left(
      \|v\|_{L^2(\Omega)} \|\nabla \cdot \bm{q} \|_{L^2(\MTh)} +
      \|\phi\|_{L^2(\Omega)} \| \nabla \times \bm{q}
      \|_{L^2(\MTh)}\right) \\
      &\leq C \left( \|v\|_{L^2(\Omega)} + \| \phi \|_{L^2(\Omega)}
      \right) \left( \|\nabla \cdot \bm{q} \|_{L^2(\MTh)} + \| \nabla
      \times \bm{q} \|_{L^2(\MTh)} \right) \\
      & \leq C \|\bm{q} \|_{L^2(\Omega)} \pnorm{\bm{q}}. \\
    \end{aligned}
  \end{displaymath}
  Moreover, we apply the trace inequality and Cauchy-Schwarz
  inequality to find 
  \begin{displaymath}
    \sum_{e \in \MEh^i} \int_e \jump{\bm{q} \cdot \un} v \d{s} \leq
    \left( \sum_{e \in \MEh^i} \int_e \frac{1}{h_e} |\jump{\bm{q}
    \cdot \un} |^2 \d{s} \right)^{1/2} \left(  \sum_{e \in \MEh^i}
    \int_e h_e |v |^2 \d{s} \right)^{1/2},   
  \end{displaymath}
  and 
  \begin{displaymath}
    h_e \|v\|_{L^2(e)}^2 \leq C \|v\|_{H^1(K)}^2, \quad e \in
    \partial K, 
  \end{displaymath}
  for any $K \in \MTh$. Hence, we have 
  \begin{displaymath}
    \sum_{e \in \MEh^i} \int_e \jump{\bm{q} \cdot \un} v \d{s} \leq C
    \pnorm{\bm{q}} \|v\|_{H^1(\Omega)} \leq C \pnorm{\bm{q}} \| \bm{q}
    \|_{L^2(\Omega)},
  \end{displaymath}
  and similarly we have the following estimate for the last term,
  \begin{displaymath}
    \begin{aligned}
      \sum_{e \in \MEh} \int_e \jump{\bm{q} \times \un} \phi \d{s}
      & \leq \left( \sum_{e \in \MEh} \int_e \frac{1}{h_e} |\jump{\bm{q}
      \times \un} |^2 \d{s} \right)^{1/2} \left(  \sum_{e \in \MEh^i}
      \int_e h_e |\phi |^2 \d{s} \right)^{1/2} \\    
      & \leq C \pnorm{\bm{q}} \|\phi\|_{H^1(\Omega)} \\ 
      & \leq C \pnorm{\bm{q}} \|\bm{q} \|_{L^2(\Omega)}. \\ 
    \end{aligned}
  \end{displaymath}
  Combining all inequalities immediately gives the estimate $\| \bm{q}
  \|_{L^2(\Omega)}^2 \leq C \pnorm{\bm{q}} \| \bm{q}
  \|_{L^2(\Omega)}$. By eliminating $\| \bm{q} \|_{L^2(\Omega)}$ we
  reach the inequality \eqref{eq:pupperbound}, which completes the
  proof.
\end{proof}

Then we claim that the bilinear form $a_h^{\bm{\mr{p}}}(\cdot, \cdot)$
is bounded and coercive with respect to the energy norm $
\pnorm{\cdot} $ for any positive $\mu$.
\begin{theorem}
  Let the coefficient $A(\bm{x}) \in L^\infty(\Omega)^{d\times d}$
  satisfy Cordes condition and let the bilinear form
  $a_h^{\bm{\mr{p}}}(\cdot, \cdot)$ be defined by \eqref{eq:bilinear}
  with any positive $\mu$, then $a_h^{\bm{\mr{p}}}(\cdot, \cdot)$
  satisfies the properties of the boundedness and coercivity:
  \begin{align}
    |a_h^{\bm{\mr{p}}} (\bm{p}, \bm{q})| & \leq C \pnorm{\bm{p}}
    \pnorm{\bm{q}}, \quad \forall \bm{p}, \bm{q} \in
    H^1(\MTh)^d, \label{eq:bounded} \\
    a_h^{\bm{\mr{p}}}(\bm{p}_h, \bm{p}_h) & \geq C \pnorm{\bm{p}_h}^2,
    \quad \forall \bm{p}_h \in \bmr{S}_h^m.
    \label{eq:coercive} 
  \end{align}
  \label{th:bounded}
\end{theorem}
\begin{proof}
  We first prove the boundedness property \eqref{eq:bounded}. Together
  with Cauchy-Schwarz inequality, one has that 
  \begin{displaymath}
    \begin{aligned}
      a_h^{\bm{\mr{p}}}(\bm{p}, \bm{q}) \leq &\left( \sum_{K \in \MTh}
      \| A: \nabla \bm{p}  \|_{L^2(K)}^2 + \sum_{e \in \MEh^i}
      \frac{\mu}{h_e} \| \jump{\bm{p} \otimes \un } \|_{L^2(e)}^2 +
      \sum_{e \in \MEh^b} \frac{\mu}{h_e} \| \bm{p} \times \un
      \|_{L^2(e)}^2 \right)^{1/2} \\
      & \left( \sum_{K \in \MTh} \| A: \nabla \bm{q}
      \|_{L^2(K)}^2 + \sum_{e \in \MEh^i} \frac{\mu}{h_e} \|
      \jump{\bm{q} \otimes \un } \|_{L^2(e)}^2 + \sum_{e \in \MEh^b}
      \frac{\mu}{h_e} \| \bm{q} \times \un \|_{L^2(e)}^2
      \right)^{1/2}.
      \\
    \end{aligned}
  \end{displaymath}
  Since $A \in L^\infty(\Omega)^{d \times d}$, we immediately get 
  \begin{displaymath}
    \|A :\nabla \bm{p} \|_{L^2(\MTh)} \leq C \| \bm{p} \|_{H^1(\MTh)},
    \quad \|A :\nabla \bm{q} \|_{L^2(\MTh)} \leq C \| \bm{q}
    \|_{H^1(\MTh)}, 
  \end{displaymath}
  which implies the estimate \eqref{eq:bounded}.

  Then we consider the term $a_h^{\bm{\mr{p}}}(\bm{p}_h, \bm{p}_h)$
  and the definition of $\pnorm{\cdot}$ indicates that it is
  sufficient to prove 
  \begin{displaymath}
    a_h^{\bm{\mr{p}}}(\bm{p}_h, \bm{p}_h) \geq C \|\nabla \bm{p}_h
    \|_{L^2(\MTh)}^2,
  \end{displaymath}
  for the coercivity of the bilinear form.  Let $\gamma$ be defined by
  \eqref{eq:gamma} and the triangle inequality shows that
  \begin{displaymath}
     | \gamma A : \nabla \bm{p}_h - \nabla \cdot \bm{p}_h |+  \gamma
     |A:\nabla \bm{p}_h| \geq |\nabla \cdot \bm{p}_h|, \quad
     \text{a.e. in } \Omega.
  \end{displaymath}
  Together with the inequality \eqref{eq:gammaAp} and $\gamma \in
  L^\infty(\Omega)$, we obtain 
  \begin{displaymath}
     \sqrt{1 - \varepsilon} | \nabla \bm{p}_h| + \|\gamma
     \|_{L^\infty(\Omega)} | A:\nabla \bm{p}_h | \geq | \nabla \cdot
     \bm{p}_h| , \quad \text{a.e. in } \Omega.
  \end{displaymath}
  By using the Cauchy-Schwarz inequality, we observe that 
  \begin{displaymath}
    \begin{aligned}
      (1 - \varepsilon)|\nabla \bm{p}_h|^2 +
      \|\gamma\|_{L^\infty(\Omega)}^2|A:\nabla \bm{p}_h|^2 +  2\sqrt{1
      - \varepsilon} | \nabla \bm{p}_h| \|\gamma \|_{L^\infty(\Omega)}
      & | A:\nabla  \bm{p}_h |  \geq |\nabla \cdot \bm{p}_h|^2 \\
      (1 - \varepsilon + C\sqrt{1 - \varepsilon}) |\nabla \bm{p}_h|^2
      + \left( \|\gamma\|_{L^\infty(\Omega)}^2 +
      \frac{\|\gamma\|_{L^\infty(\Omega)}}{C} \right) |A:\nabla
      \bm{p}_h|^2 & \geq  |\nabla \cdot \bm{p}_h|^2 \quad \text{a.e.
      in } \Omega, \\
    \end{aligned}
  \end{displaymath}
  for any $C > 0$. Since $1 - \varepsilon < 1$, we take a
  proper $C > 0$ such that there exist two constants $0 < C_1 < 1$,
  $C_2 > 0$ satisfying
  \begin{displaymath}
    (1 - C_1) |\nabla \bm{p}_h|^2 + C_2 |A:\nabla \bm{p}_h|^2  \geq
    |\nabla \cdot \bm{p}_h|^2 \quad \text{a.e.  in } \Omega.
  \end{displaymath}
  Integration over all elements gives us that 
  \begin{displaymath}
    (1 - C_1) \sum_{K \in \MTh} \|\nabla \bm{p}_h \|_{L^2(K)}^2 + C_2
    \sum_{K \in \MTh} \| A:\nabla \bm{p}_h\|_{L^2(K)}^2 \geq \sum_{K
    \in \MTh} \|\nabla \cdot \bm{p}_h \|_{L^2(K)}^2.
  \end{displaymath}
  By the estimate \eqref{eq:discreteMT}, we first select a
  sufficiently large $\mu$ to derive
  \begin{displaymath}
    \begin{aligned}
      (1 - C_1) \sum_{K \in \MTh} \|\nabla \bm{p}_h \|_{L^2(K)}^2 & +
      C_2 \sum_{K \in \MTh} \| A:\nabla \bm{p}_h\|_{L^2(K)}^2 +
      \sum_{e \in \MEh^i} \frac{C_2 \mu}{h_e} \| \jump{\bm{q} \otimes
      \un } \|_{L^2(e)}^2  \\ 
      + \sum_{e \in \MEh^b} \frac{C_2 \mu}{h_e} \| \bm{q} \times \un
      \|_{L^2(e)}^2 \geq& \sum_{K \in \MTh} \|\nabla \cdot \bm{p}_h
      \|_{L^2(K)}^2 +  \sum_{e \in \MEh^i} \frac{C_3}{h_e} \|
      \jump{\bm{q} \otimes \un } \|_{L^2(e)}^2    + \sum_{e \in
      \MEh^b} \frac{C_3}{h_e} \| \bm{q} \times \un \|_{L^2(e)}^2
      \\
      \geq & \sum_{K \in \MTh} \|\nabla \bm{p}_h \|_{L^2(K)}^2, \\
    \end{aligned}
  \end{displaymath}
  which actually yields 
  \begin{displaymath}
    a_h^{\bm{\mr{p}}}(\bm{p}_h, \bm{p}_h) \geq C  \sum_{K \in \MTh}
    \|\nabla \bm{p}_h \|_{L^2(K)}^2.
  \end{displaymath}
  With sufficiently large $\mu$, we have proven the coercivity
  \eqref{eq:coercive}. Note that by scaling arguments we conclude that
  for any positive $\mu$ the coercivity still holds, which completes
  the proof.
\end{proof}

We have established the existence and uniqueness of the solution to
the minimization problem \eqref{eq:infp} or equivalently to the
problem \eqref{eq:varp}. Then let us firstly give {\it a priori} error
estimate of the method proposed for seeking an approximation to the
gradient $\bm{p}$ in \eqref{eq:upnondiv}
\begin{theorem}
  Let $\bm{p} \in \bmr{I}^{m+1}(\Omega)$ be the solution to
  \eqref{eq:upnondiv} and let $\bm{p}_h \in \bmr{S}_h^m$ be the
  solution to \eqref{eq:varp}. Let the coefficient $A(\bm{x}) \in
  L^\infty(\Omega)^{d \times d}$ satisfy the Cordes condition, then
  the following estimate holds:
  \begin{equation}
    \pnorm{\bm{p} - \bm{p}_h} \leq Ch^m \|\bm{p} \|_{H^{m+1}(\Omega)}.
    \label{eq:dgerror}
  \end{equation}
  \label{th:error}
\end{theorem}
\begin{proof}
  The orthogonal property directly follows from the
  definitions of the bilinear form $a_h^{\bm{\mr{p}}}(\cdot, \cdot)$
  and linear form $l_h^{\bm{\mr{p}}}(\cdot)$: for any $\bm{q}_h \in
  \bmr{S}_h^m$, one has that
  \begin{displaymath}
    a_h^{\bm{\mr{p}}}(\bm{p} - \bm{p}_h, \bm{q}_h) = 0.
  \end{displaymath}
  Then for any $\bm{q}_h \in \bmr{S}_h^m$, together with the
  boundedness \eqref{eq:bounded} and coercivity \eqref{eq:coercive},
  there holds
  \begin{displaymath}
    \begin{aligned}
      \pnorm{ \bm{p}_h - \bm{q}_h }^2 & \leq C a_h^{\bm{\mr{p}}}(
      \bm{p}_h - \bm{q}_h, \bm{p}_h - \bm{q}_h) \\
      & =  C a_h^{\bm{\mr{p}}}( \bm{p} - \bm{q}_h, \bm{p}_h -
      \bm{q}_h) \\
      & \leq C \pnorm{\bm{p} - \bm{q}_h} \pnorm{ \bm{p}_h - \bm{q}_h}.
    \end{aligned}
  \end{displaymath}
  By eliminating $ \pnorm{ \bm{p}_h - \bm{q}_h}$, together with the
  triangle inequality, we observe that 
  \begin{displaymath}
    \begin{aligned}
      \pnorm{ \bm{p}_h - \bm{q}_h } & \leq C  \pnorm{\bm{p} -
      \bm{q}_h} \\
      \pnorm{ \bm{p}_h - \bm{q}_h } +  \pnorm{ \bm{p} - \bm{q}_h }
      & \leq C \pnorm{\bm{p} - \bm{q}_h} \\
      \pnorm{ \bm{p} - \bm{p}_h } &\leq C \inf_{\bm{q}_h \in
      \bmr{S}_h^m} \pnorm{\bm{p} - \bm{q}_h } \leq C \pnorm{\bm{p} -
      \Pi_h^{\bmr{S}, m} \bm{p}}. \\
    \end{aligned}
  \end{displaymath}
  From Lemma \ref{le:globalL2interpolation}, it is easy to deduce that 
  \begin{displaymath}
    \begin{aligned}
      \|\nabla(\bm{p} -  \Pi_h^{\bmr{S}, m} \bm{p})\|_{L^2(K)} & \leq
      C h^m \|\bm{p} \|_{H^{m+1}(K)}, \quad \forall K \in \MTh, \\
      h_e^{-1/2} \| \jump{ (\bm{p} -  \Pi_h^{\bmr{S}, m} \bm{p})
      \otimes \un} \|_{L^2(e)} & \leq C h^m \| \bm{p}
      \|_{H^{m+1}(K)}, \quad \forall e \in \partial K, \quad \forall K
      \in \MTh. \\
    \end{aligned}
  \end{displaymath}
  Hence, we conclude that
  \begin{displaymath}
    \pnorm{ \bm{p} - \bm{p}_h } \leq C \pnorm{\bm{p} -
    \Pi_h^{\bmr{S}, m} \bm{p}} \leq C h^m \| \bm{p} \|_{H^{m+1}
    (\Omega)},
  \end{displaymath}
  which gives us the estimate \eqref{eq:dgerror} and completes the
  proof.
\end{proof}
%As we emphesize in Remark \ref{re:order}, the method also work
%in the case $d = 3$ and $m \geq 3$, as demonstrated in Section
%\ref{sec:numericalresults}. 

Until now, we have developed a discontinuous least squares finite
element method to get a numerical approximation to the variable $\bm
p$ in the system \eqref{eq:upnondiv}. After that, we propose another
least squares finite element method to obtain an approximation to $u$.
We introduce a least squares functional $J_h^{\bm{\mr{u}}}(\cdot;
\cdot)$ defined by
\begin{equation}
  J_h^{\bm{\mr{u}}}(v; \bm{q}_h ) := \sum_{K \in \MTh} \|\nabla v -
  \bm{q}_h \|_{L^2(K)}^2 + \sum_{e \in \MEh^b} \frac{1}{h} \|v -
  g\|_{L^2(e)}^2,
  \label{eq:ufunc}
\end{equation}
where $g$ is the boundary condition in \eqref{eq:upnondiv}. 
We minimize the functional \eqref{eq:ufunc} on the standard $C^0$
finite element space $ \wt{V}_h^m := V_h^m \cap H^1(\Omega)$, together
with the numerical gradient, to get a numerical approximation $u_h$.
Precisely, the minimization problem reads
\begin{equation}
  \inf_{v_h \in \wt{V}_h^m} J_h^{\bm{\mr{u}}}(v_h; \bm{p}_h),
  \label{eq:infu}
\end{equation}
where $\bm{p}_h$ is the solution to \eqref{eq:infp}. We write the
Euler-Lagrange equation to solve the problem \eqref{eq:infu} and the
corresponding variational problem takes the form: find $u_h \in
\wt{V}_h^m$ such that
\begin{equation}
  a_h^{\bm{\mr{u}}}(u_h, v_h) = l_h^{\bm{\mr{u}}}(v_h; \bm{p}_h),
  \quad \forall v_h \in \wt{V}_h^m,
  \label{eq:varu}
\end{equation}
where the bilinear form $a_h^{\bm{\mr{u}}}(\cdot, \cdot)$ is defined
as
\begin{displaymath}
  a_h^{\bm{\mr{u}}}(u_h, v_h) = \sum_{K \in \MTh} \int_K \nabla u_h
  \cdot \nabla v_h \d{x} + \sum_{e \in \MEh^b} \frac{1}{h} \int_e
  u_h v_h \d{s},
\end{displaymath}
and the linear form $l_h^{\bm{\mr{u}}}(\cdot; \cdot)$ is defined as
\begin{displaymath}
  l_h^{\bm{\mr{u}}}(v_h; \bm{p}_h) = \sum_{K \in \MTh} \int_K \nabla
  v_h \cdot \bm{p}_h \d{x} + \sum_{e \in \MEh^b} \frac{1}{h} \int_e
  v_h g \d{s}.
\end{displaymath}
Let us define a natural energy norm $\unorm{\cdot}$ from the bilinear
form $a_h^{\bmr{u}}(\cdot, \cdot)$:
\begin{displaymath}
  \unorm{v}^2 := \sum_{K \in \MTh} \| \nabla v \|_{L^2(K)}^2 + \sum_{e
  \in \MEh^b} \frac{1}{h} \|v\|_{L^2(e)}^2,
\end{displaymath}
for any $v \in H^1(\Omega)$. Note that $\unorm{v}^2 =
a_h^{\bm{\mr{u}}}(v, v)$ for any $v \in H^1(\Omega)$. Indeed we only
need to prove that $\unorm{\cdot}$ is actually a norm on the space
$H^1(\Omega)$ and the existence and uniqueness of the solution to
\eqref{eq:varu} are then the direct consequences.
\begin{lemma}
  $\unorm{\cdot}$ is a norm on the space $H^1(\Omega)$.
  \label{le:uisanorm}
\end{lemma}
\begin{proof}
  It is sufficient to prove that $\unorm{v} = 0$ indicates $v = 0$. If
  $\unorm{v} = 0$ for some $v \in H^1(\Omega)$, we have that 
  \begin{displaymath}
    \nabla v = 0 \quad \text{in } \Omega, \quad \text{and} \quad v = 0
    \quad \text{on } \partial \Omega,
  \end{displaymath}
  which gives us that $v = 0$ on the whole domain and completes the
  proof.
\end{proof}

%\begin{lemma}
  %For any $v \in H^1(\Omega)$, the following estimate holds:
  %\begin{equation}
    %\|v\|_{H^1(\Omega)} \leq C \unorm{v}.
    %\label{eq:unormlbound}
  %\end{equation}
  %\label{th:unormlbound}
%\end{lemma}
%\begin{proof}
  %It is sufficient to prove $\|v\|_{L^2(\Omega)} \leq C \unorm{v}$.
  %Define $\phi \in H^2(\Omega) \cap H^1_0(\Omega)$ by $ - \Delta \phi
  %= v$ and the regularity estimate $\| \phi \|_{H^2(\Omega)} \leq \| v
  %\|_{L^2(\Omega)}$ holds. By integration by parts, we observe 
  %\begin{displaymath}
    %\begin{aligned}
      %& \|v\|^2_{L^2(\Omega)} = \sum_{K \in \MTh} \int_K - \Delta \phi
      %v \d{x} = \sum_{K \in \MTh} \int_K \nabla \phi \cdot \nabla v
      %\d{x} - \sum_{e \in \MEh^b} \int_e v \nabla \phi \cdot \un \d{x}
      %\\
      %& \leq \left(  \sum_{K \in \MTh} \|\nabla v\|_{L^2(K)} + \sum_{e
      %\in \MEh^b} h_e^{-1} \|v\|_{L^2(e)}^2  \right)  \left( \sum_{K
      %\in \MTh} \|\nabla \phi \|_{L^2(K)} + \sum_{e \in \MEh^b} h_e \|
      %\phi \|_{L^2(e)}^2  \right) \\
      %& \leq \unorm{v} \|\phi\|_{H^2(\Omega)} \leq \unorm{v}
      %\|v\|_{L^2(\Omega)}.\\
    %\end{aligned}
  %\end{displaymath}
  %The last inequality follows the trace inequality, which completes
  %the proof.
%\end{proof}
With respect to the energy norm $\unorm{\cdot}$, we have the
following error estimate.
\begin{theorem}
  Let $u \in H^{m+1}(\Omega)$ be the solution to \eqref{eq:upnondiv}
  and let $u_h \in \wt{V}_h^m$ be the solution to \eqref{eq:varu},
  then the following estimate holds:
  \begin{equation}
    \unorm{u - u_h} \leq C \left( h^m \|u\|_{H^{m+1}(\Omega)} + \|
    \bm{p} - \bm{p}_h\|_{L^2(\Omega)} \right),
    \label{eq:dguerror}
  \end{equation}
  where $\bm{p}_h$ is the solution to \eqref{eq:varp}.
  \label{th:dguerror}
\end{theorem}
\begin{proof}
  Let $u_I \in \wt{V}_h^m$ be the interpolant of $u$ and we deduce
  that 
  \begin{displaymath}
    \begin{aligned}
%      \unorm{u - u_h}^2 = J_h^{\bm{\mr{u}}}(u_h; \nabla u) \leq
      %J_h^{\bm{\mr{u}}}(u_I) &\leq \sum_{K \in \MTh} \|\nabla u_I -
      %\nabla u + \bm{p} - \bm{p}_h \|_{L^2(K)}^2 + \sum_{e \in \MEh^b}
      %h_e^{-1} \|u_I - g \|_{L^2(e)}^2 \\
      %& \leq C \left( \unorm{u - u_I}^2 + \|\bm{p} -
      %\bm{p}_h\|_{L^2(\Omega)}^2  \right).
      \unorm{u - u_h}^2 & = \sum_{K \in \MTh} \|\nabla u - \nabla u_h
      \|_{L^2(K)}^2 + \sum_{e \in \MEh^b} \frac{1}{h} \| u - u_h
      \|_{L^2(e)}^2 \\
      & \leq C \left(  \sum_{K \in \MTh} \|\nabla u - \bm{p}_h
      \|_{L^2(K)}^2 + \sum_{e \in \MEh^b} \frac{1}{h} \| u - u_h
      \|_{L^2(e)}^2 + \|\bm{p} - \bm{p}_h \|_{L^2(\Omega)} \right) \\
      & \leq C \left( J_h^{\bmr{u}}(u_h; \bm{p}_h) + \|\bm{p} -
      \bm{p}_h \|_{L^2(\Omega)} \right) \leq C \left(
      J_h^{\bmr{u}}(u_I; \bm{p}_h) + \|\bm{p} - \bm{p}_h
      \|_{L^2(\Omega)} \right) \\
      & \leq C \left(  \sum_{K \in \MTh} \|\nabla u - \nabla u_I
      \|_{L^2(K)}^2 + \sum_{e \in \MEh^b} \frac{1}{h} \| u - u_I
      \|_{L^2(e)}^2 + \|\bm{p} - \bm{p}_h \|_{L^2(\Omega)} \right) \\
      & \leq C \left( \unorm{u - u_I} +  \|\bm{p} - \bm{p}_h
      \|_{L^2(\Omega)} \right). \\
    \end{aligned}
  \end{displaymath}
  By the trace inequality, it is trivial to obtain 
  \begin{displaymath}
    \unorm{ u - u_I} \leq C h^m \|u\|_{H^{m+1}(\Omega)},
  \end{displaymath}
  which implies \eqref{eq:dguerror} and completes the proof.
\end{proof}
Then we attain an error estimate with respect to the $L^2$-norm.
\begin{theorem}
  Let $u \in H^{m+1}(\Omega)$ be the solution to \eqref{eq:upnondiv}
  and let $u_h \in \wt{V}_h^m$ be the solution to \eqref{eq:varu},
  then the following estimate holds:
  \begin{equation}
    \| u - u_h \|_{L^2(\Omega)} \leq C \left(h^{m+1} \| u
    \|_{H^{m+1}(\Omega)}  + \|\bm{p} - \bm{p}_h \|_{L^2(\Omega)}
    \right),
    \label{eq:L2uerror}
  \end{equation}
  where $\bm{p}_h$ is the solution to \eqref{eq:varp}.
  \label{th:L2uerror}
\end{theorem}
\begin{proof}
  Let $e_h = u - u_h$ and by the direct calculation we could see that
  \begin{displaymath}
    a_h^{\bm{\mr{u}}}(e_h, v_h) = (\bm{p} - \bm{p}_h, \nabla
    v_h)_{L^2(\Omega)}, \quad \forall v_h \in \wt{V}_h^m.
  \end{displaymath}
  We take $\psi = e_h$ and we let $w \in H^2(\Omega) \cap
  H^1_0(\Omega)$ be the unique solution to the problem $ -\Delta w =
  \psi$.  We denote $w_I \in \wt{V}_h^m$ as the linear interpolant of
  $w$. One can observe that
  \begin{displaymath}
    \begin{aligned}
      (e_h, \psi)_{L^2(\Omega)} & = (\nabla e_h, \nabla
      w)_{L^2(\Omega)} - \left( e_h, \frac{\partial w}{\partial \un}
      \right)_{L^2(\partial \Omega)}\\
      &= a_h^{\bm{\mr{u}}}(e_h, w - w_I) + (\bm{p} - \bm{p}_h, \nabla
      w_I) -  \left( e_h, \frac{\partial w}{\partial \un}
      \right)_{L^2(\partial \Omega)} \\ 
      & \leq C \unorm{e_h} \unorm{w - w_I} + \|\bm{p} -
      \bm{p}_h\|_{L^2(\Omega)} \|\nabla w_I\|_{L^2(\Omega)} +
      \|e_h\|_{H^{-1/2}(\partial \Omega)} \left\| \frac{\partial
      w}{\partial \un} \right\|_{H^{3/2}(\partial \Omega)} \\
      &\leq C \left( h^{m+1} \|u\|_{H^{m+1}(\Omega)} + \|\bm{p} -
      \bm{p}_h\|_{L^2(\Omega)} + \|e_h\|_{H^{-1/2}(\partial \Omega)}
      \right) \|w\|_{H^2(\Omega)}.
    \end{aligned}
  \end{displaymath}
  Together with the regularity estimate $\|w\|_{H^2(\Omega)} \leq C
  \| \psi \|_{L^2(\Omega)}$, we immediately get 
  \begin{equation}
    \|e_h\|_{L^2(\Omega)} \leq C  \left( h^{m+1}
    \|u\|_{H^{m+1}(\Omega)} + \|\bm{p} - \bm{p}_h\|_{L^2(\Omega)} +
    \|e_h\|_{H^{-1/2}(\partial \Omega)} \right).
    \label{eq:ehL2upper}
  \end{equation}
  We end the proof by giving a bound for the term $\|e_h
  \|_{H^{-1/2}(\partial \Omega)}$, which is defined by
  \begin{displaymath}
    \|e_h \|_{H^{-1/2}(\partial \Omega)} = \sup_{\psi \in
    H^{1/2}(\partial \Omega)} \frac{(e_h, \psi)_{L^2(\partial \Omega)
    }}{ \| \psi\|_{H^{1/2}(\partial \Omega)}}.
  \end{displaymath}
  For any $\tau \in H^{1/2}(\partial \Omega)$, we let $\alpha \in
  H^1(\Omega)$ solve the problem 
  \begin{displaymath}
    - \Delta \alpha = 0, \quad \text{in } \Omega, \quad \alpha = \tau,
    \quad \text{on } \partial \Omega.
  \end{displaymath}
  We denote by $\alpha_I \in \wt{V}_h^m$ the interpolant of $\alpha$,
  then we could obtain that
  \begin{displaymath}
    \begin{aligned}
      (e_h, \tau)_{L^2(\partial \Omega)} & = h\left( \sum_{e \in
      \MEh^b} \frac{1}{h} \int_e e_h \alpha \d{s} \right) \\
      & = h \left( \sum_{e \in \MEh^b} \frac{1}{h} \int_e e_h \alpha
      \d{s} - a_h^{\bmr{u}}(e_h, \alpha_I)  \right) + h(\bm{p} -
      \bm{p}_h, \nabla \alpha_I)_{L^2(\Omega)} \\
      & = h \left( \sum_{e \in \MEh^b}\frac{1}{h} \int_e e_h ( \alpha
      - \alpha_I)  \d{s} - \sum_{K \in \MTh} \int_K \nabla e_h \cdot
      \nabla \alpha_I \d{x} \right)+ h(\bm{p} - \bm{p}_h, \nabla
      \alpha_I)_{L^2(\Omega)}. \\
    \end{aligned}
  \end{displaymath}
  Applying the Cauchy-Schwarz inequality and the approximation result
  of $\alpha_I$, we obtain that 
  \begin{displaymath}
    \begin{aligned}
      (e_h, \tau)_{L^2(\partial \Omega)} & \leq Ch \unorm{e_h} \left(
      \|\alpha - \alpha_I\|_{L^2(\partial \Omega)} +
      \|\nabla \alpha_I\|_{L^2(\Omega)}  \right) + C h\|\bm{p} -
      \bm{p}_h \|_{L^2(\Omega)} \|\nabla \alpha_I\|_{L^2(\Omega)} \\
      & \leq Ch\left( \unorm{e_h} + \|\bm{p} - \bm{p}_h
      \|_{L^2(\Omega)} \right) \|\alpha \|_{H^1(\Omega)}. \\
    \end{aligned}
  \end{displaymath}
  Together with the regularity inequality $\|\alpha \|_{H^1(\Omega)}
  \leq C \|\tau \|_{H^{1/2}(\partial \Omega)}$, we give a bound of
  $ \|e_h\|_{H^{-1/2}(\partial \Omega)}$,
  \begin{displaymath}
    \|e_h\|_{H^{-1/2}(\partial \Omega)} \leq Ch\left( \unorm{e_h} +
    \|\bm{p} - \bm{p}_h \|_{L^2(\Omega)} \right).
  \end{displaymath}
  Combining \eqref{eq:dgerror} and \eqref{eq:ehL2upper} implies the
  estimate \eqref{eq:L2uerror} and completes the proof.
\end{proof}
%\begin{remark}
  %The Theorem \ref{th:error} restricts $1 \leq m \leq 2$ for the case
  %$d = 3$, as required in Lemma \ref{le:discreteMT}. As we emphasize
  %in Remark \ref{re:order}, the numerical results demonstrate that our
  %method could also work for the case $m \geq 3$.
%\end{remark}
\begin{remark}
  The optimal convergence order of $\| u - u_h\|_{L^2(\Omega)}$
  depends on the convergence order of the term
  $ \| \bm{p} - \bm{p}_h\|_{L^2(\Omega)}$. We can only prove a
  suboptimal $L^2$ convergence rate for the variable
  $\bm{p}$. However, the numerical results in next section demonstrate
  our proposed method produces an approximation for $\bm{p}$ with an
  optimal $L^2$ convergence rate.  Actually, when one degree higher
  polynomials are employed to approximate $\bm{p}$, it is clear that
  the error $\| u - u_h\|_{L^2(\Omega)}$ will converge optimally
  from Theorem \ref{th:error} and Theorem \ref{th:L2uerror}.
\end{remark}

Since the solution to problem \eqref{eq:nondiv} may be of low
regularity, we note that the least squares functional we try to
minimize can automatically serves as an {\it a posteriori} error
estimator. Precisely, we define the element estimator
$\eta_K(\bm{p}_h)$ as
\begin{equation}
  \begin{aligned}
    \eta_K^2(\bm{p}_h) := \| A: \nabla \bm{p}_h - f\|_{L^2(K)}^2 + \|
    h_e^{-1/2} \jump{\bm{p}_h \otimes \un} \|_{L^2(\partial K \cap
    \MEh^i)}^2  + \| h_e^{-1/2}  ( \bm{p}_h \times \un)
    \|_{L^2(\partial K \cap \MEh^b)}^2.
  \end{aligned}
  \label{eq:etadef}
\end{equation}
%As a direct consequences of former results, we have the following
%estimate:
%\begin{corollary}
  %Let $\bm{p}$ be the solution to \eqref{eq:upnondiv} and let
  %$\bm{p}_h \in \bmr{S}_h^m$ and let $\eta_K$ be piecewisely defined
  %as \eqref{eq:etadef}, then the following inequality holds
  %\begin{equation}
    %C_1 \pnorm{\bm{p} - \bm{p}_h}^2 \leq  \sum_{T \in \MTh}
    %\eta_K^2(\bm{p}_h) \leq C_2 \pnorm{\bm{p} - \bm{p}_h}^2.
    %\label{eq:estimator}
  %\end{equation}
  %\label{th:estimator}
%\end{corollary}
We adopt the longest-edge bisection algorithm to avoid the hanging
nodes.  To close this section, we outline the following adaptive
algorithm:
\begin{itemize}
  \item[\bf{Initialize}] Given the initial mesh $\mc{T}_0$ and a
    parameter $0 < \theta < 1$. Set $l=0$.
  \item[\bf{Solve}] Solve and obtain the numerical solution with
    respect to the mesh $\mc{T}_l$.
  \item[\bf{Estimate}] Compute the error estimator $\eta_K$ on all
    elements in $\mc{T}_l$.
  \item[\bf{Mark}] Construct the minimal subset $\mc{M} \subset
    \mc{T}_l$ such that $\theta \sum_{K \in \mc{T}_l} \eta_K^2 \leq
    \sum_{K \in \mc{M}} \eta_K^2$ and mark all elements in $\mc{M}$.
  \item[\bf{Refine}] Refine all elements in $\mc{M}$ and generate a
    conforming mesh $\mc{T}_{l+1}$ from $\mc{T}_l$. Set $l = l + 1$
    and repeat the loop.
\end{itemize}

%%% Local Variables:
%%% mode: latex
%%% TeX-master: "nondivergence"
%%% End:

% vim:spell:tw=70:fo+=Mn:cc=70
\section{Numerical Results}
\label{sec:numericalresults}
In this section, we carry out a series of numerical experiments to
demonstrate the convergence rates predicted by theoretical analysis in
section \ref{sec:dlsm}. In all cases, the parameter $\eta$ in the
bilinear form $a_h^{\bm{\mr{p}}}(\cdot, \cdot)$ is taken as $10$. 

\begin{figure}
  \centering
  \includegraphics[width=0.4\textwidth]{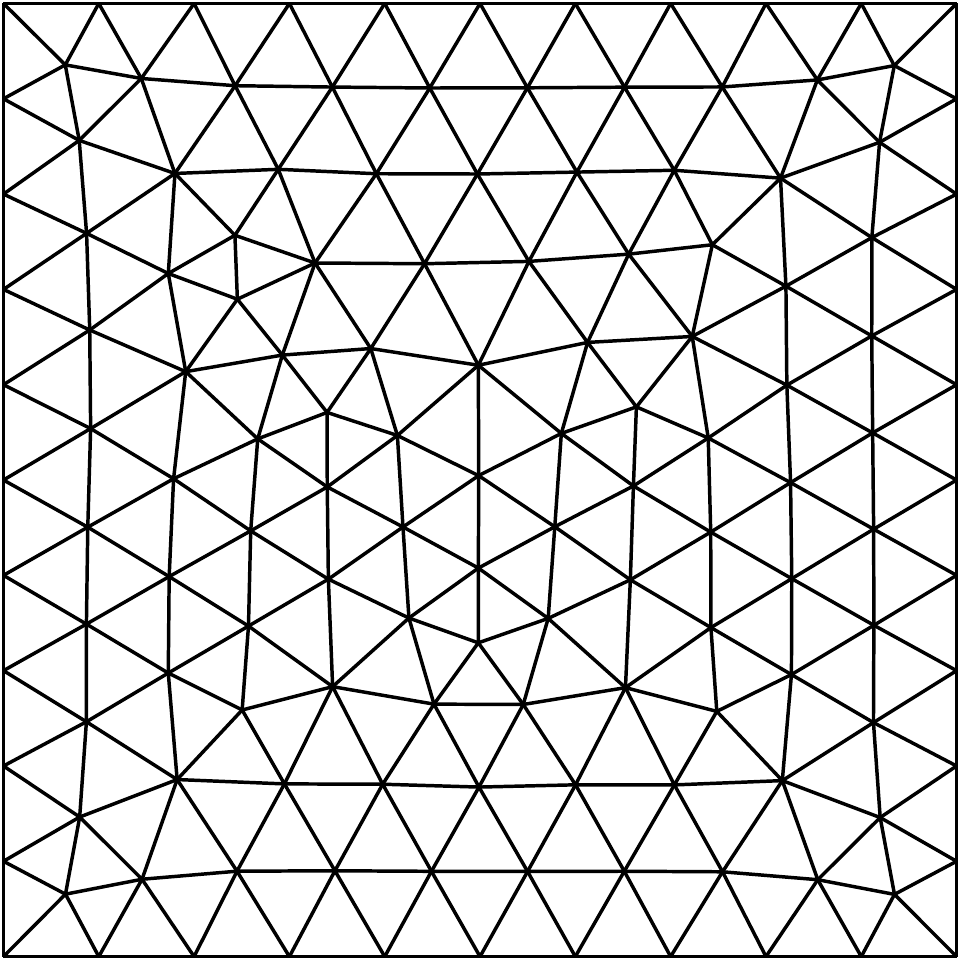}
  \hspace{25pt}
  \includegraphics[width=0.4\textwidth]{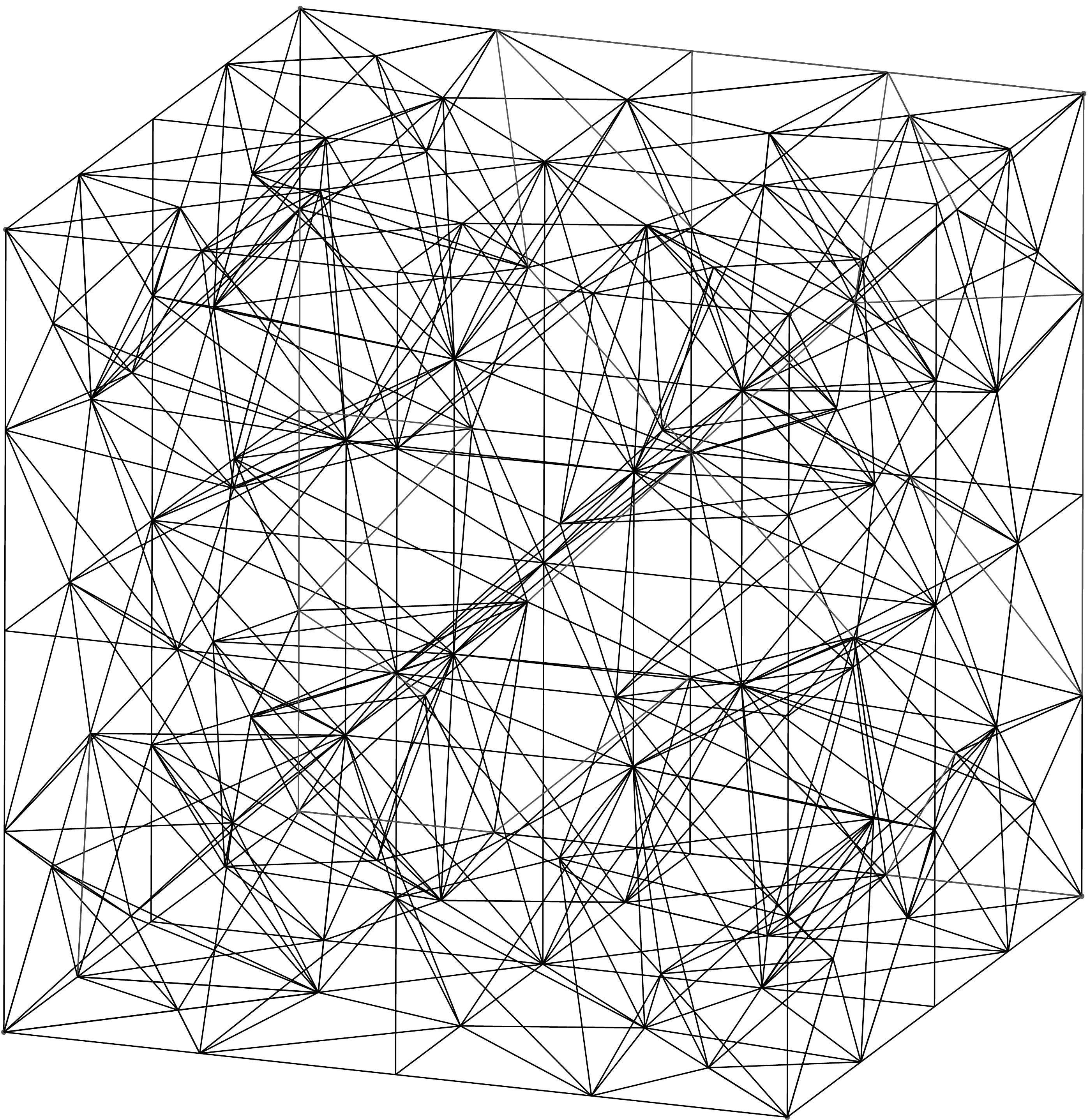}
  \caption{2d triangular partition (left) / 3d tetrahedral partition
  (right).}
  \label{fig:partition}
\end{figure}

\noindent \textbf{Example 1.} In the first example, we consider a
smooth problem in two dimensions. On the domain $\Omega = (-1, 1)^2$,
we select the exact solution $u(x, y)$ and the smooth coefficient
$A(x, y)$ as
\begin{displaymath}
  u(x, y) = xy\sin(2\pi x) \sin(3 \pi y), \quad (x, y) \text{ in }
  \Omega,
\end{displaymath}
and 
\begin{displaymath}
  A(x, y) = \begin{bmatrix}
    | \sin(4\pi x)|^{1/5} + 1 & \cos(2xy \pi) \\
    \cos(2xy \pi) &  | \sin(4\pi y)|^{1/5} + 1 \\
  \end{bmatrix}.
\end{displaymath}
The source term and boundary condition are taken accordingly. We solve
this problem on a sequence of triangular meshes with mesh size $h =
1/10, 1/20, \cdots, 1/160$, see Fig~\ref{fig:partition} for the
coarsest mesh.  We employ the finite element spaces $\bmr{S}_h^m \times
\wt{V}_h^m$ with $1 \leq m \leq 3$ to seek numerical solutions
$(\bm{p}_h, u_h)$ for approximating $(\bm{p}, u)$ in
\eqref{eq:upnondiv}. For the gradient, we plot the errors
$\pnorm{\bm{p} - \bm{p}_h}$ and $\| \bm{p} - \bm{p}_h\|_{L^2(\Omega)}$
in Fig~\ref{fig:ex1perror}. For fixed $m$, it is clear that the error
$\pnorm{\bm{p} - \bm{p}_h}$ converges to zero with the rate $O(h^m)$
and error $\| \bm{p} - \bm{p}_h\|_{L^2(\Omega)}$ converges to zero
with the rate $O(h^{m+1})$ as the mesh size decreases to zero. All
convergence rates are optimal and coincide with the Theorem
\ref{th:error}. For $u$, we plot the numerical errors in both $L^2$
norm and energy norm in Fig~\ref{fig:ex1uerror}. We also attain the
optimal convergence rates $O(h^m)$ and $O(h^{m+1})$ for the errors
$\unorm{u - u_h}$ and $\| u - u_h\|_{L^2(\Omega)}$, respectively. We
note that all the convergence rates perfectly agree with the error
estimates.

%\begin{figure}
  %\centering
  %\includegraphics[width=0.4\textwidth]{./figure/tri1-crop.pdf}
  %\hspace{25pt}
  %\includegraphics[width=0.4\textwidth]{./figure/tri2-crop.pdf}
  %\caption{Triangulation with mesh size $h=1/10$ (left)
  %/ $h=1/20$ (right).}
  %\label{fig:triangulation}
%\end{figure}

\begin{figure}[htb]
  \centering
  \includegraphics[width=0.4\textwidth]{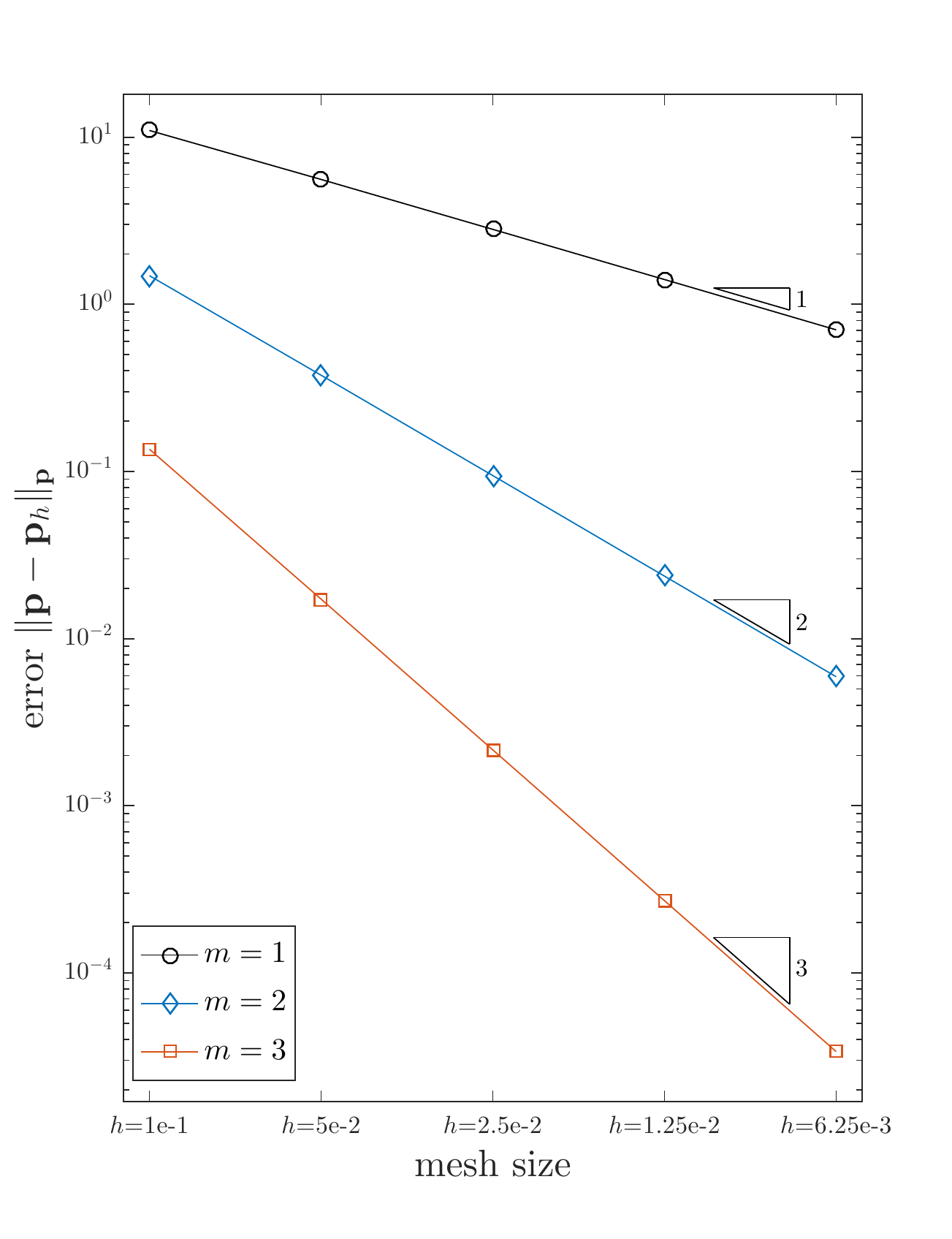}
  \includegraphics[width=0.4\textwidth]{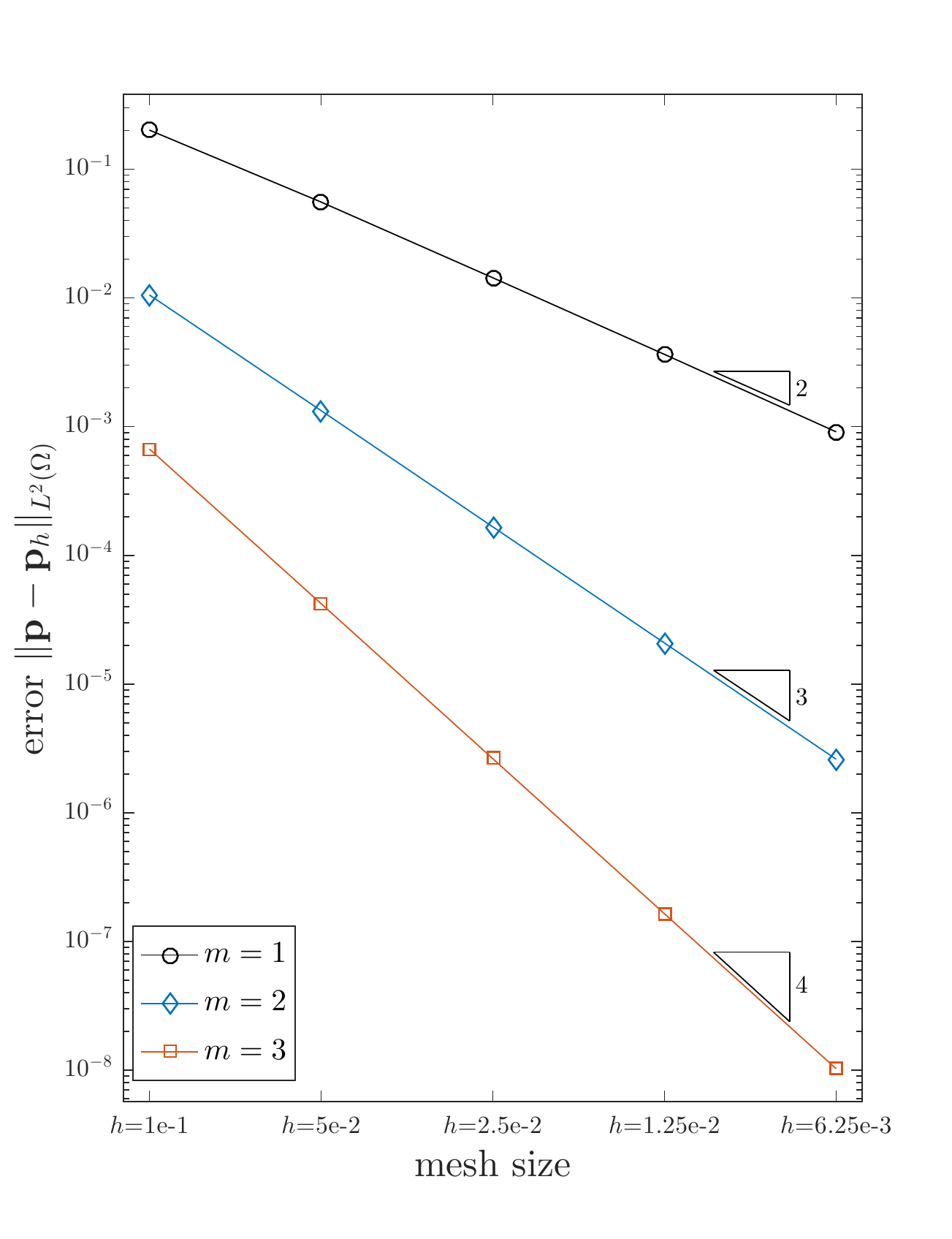}
  \caption{Example 1. The convergence rates of $\pnorm{\bm{p} -
  \bm{p}_h}$ (left) / $\| \bm{p} - \bm{p}_h\|_{L^2(\Omega)}$ (right).
  }
  \label{fig:ex1perror}
\end{figure}

\begin{figure}[htb]
  \centering
  \includegraphics[width=0.4\textwidth]{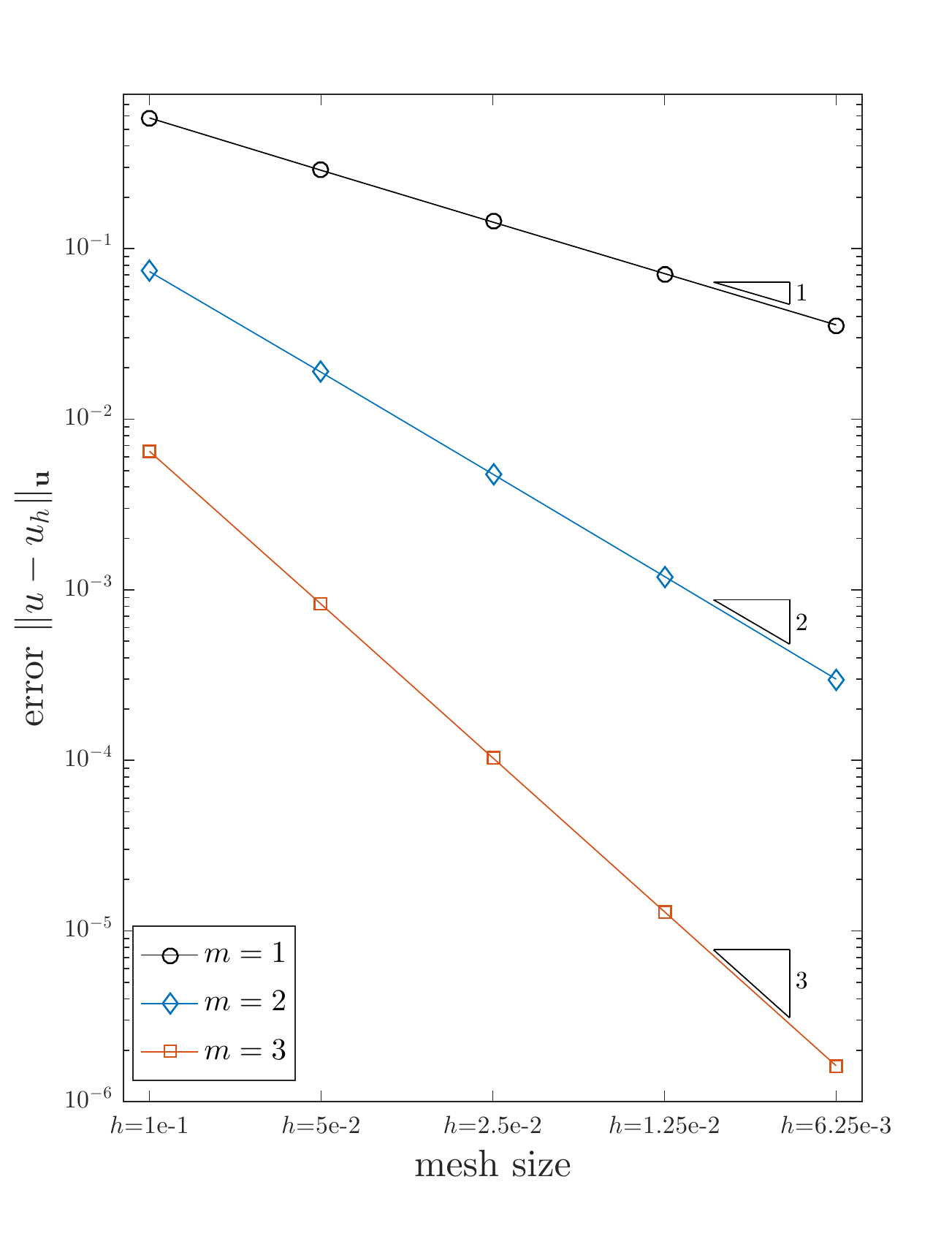}
  \includegraphics[width=0.4\textwidth]{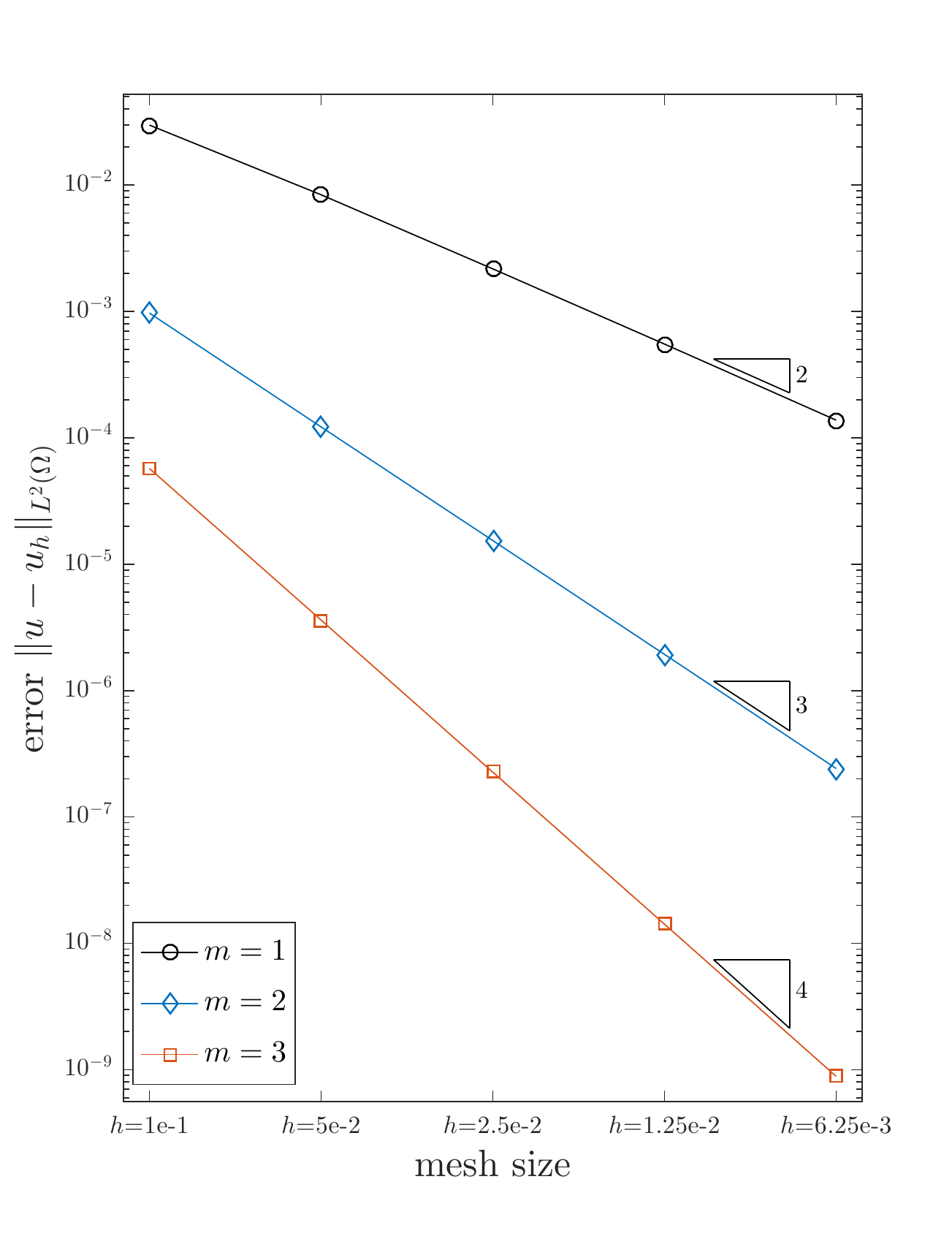}
  \caption{Example 1. The convergence rates of $\unorm{u - u_h}$
  (left) / $\| u - u_h\|_{L^2(\Omega)}$ (right).  }
  \label{fig:ex1uerror}
\end{figure}

\noindent \textbf{Example 2.} In this example, we choose a
discontinuous coefficients $A(x, y)$ which reads
\begin{displaymath}
  A(x, y) = \begin{bmatrix}
    2 & \frac{xy}{|xy|} \\
    \frac{xy}{|xy|} & 2 \\
  \end{bmatrix}.
\end{displaymath}
The exact solution and the triangular meshes and the approximating
spaces are taken as the same as in Example 1. The numerically
convergence rates are displayed in Fig~\ref{fig:ex2perror} and
Fig~\ref{fig:ex2uerror}. Clearly, for both variables $\bm{p}$ and $u$
the rates of convergence in $L^2$ norm and energy norm are $m+1$ and
$m$, respectively, which again are in perfect agreement with
theoretical results.

\begin{figure}[htb]
  \centering
  \includegraphics[width=0.4\textwidth]{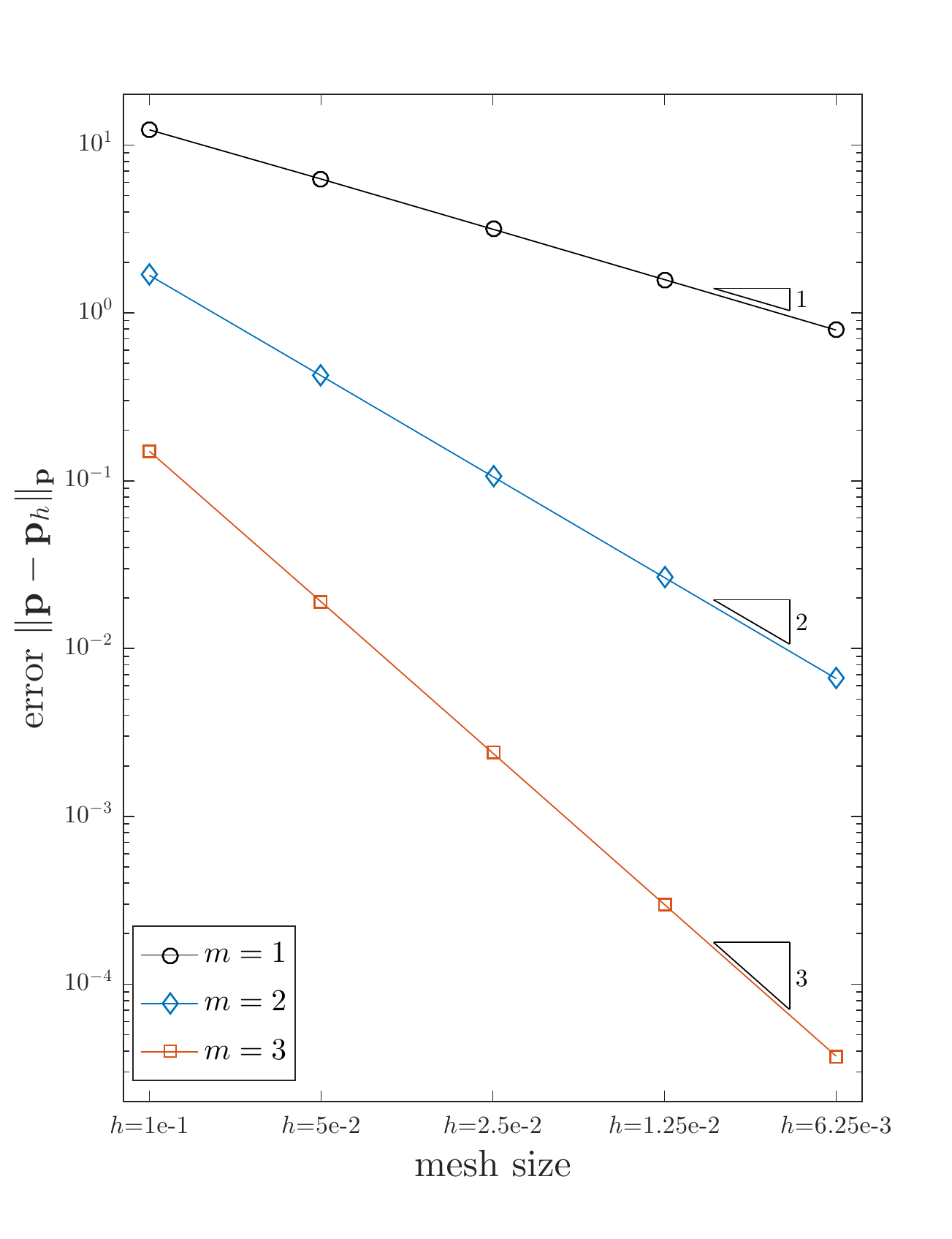}
  \includegraphics[width=0.4\textwidth]{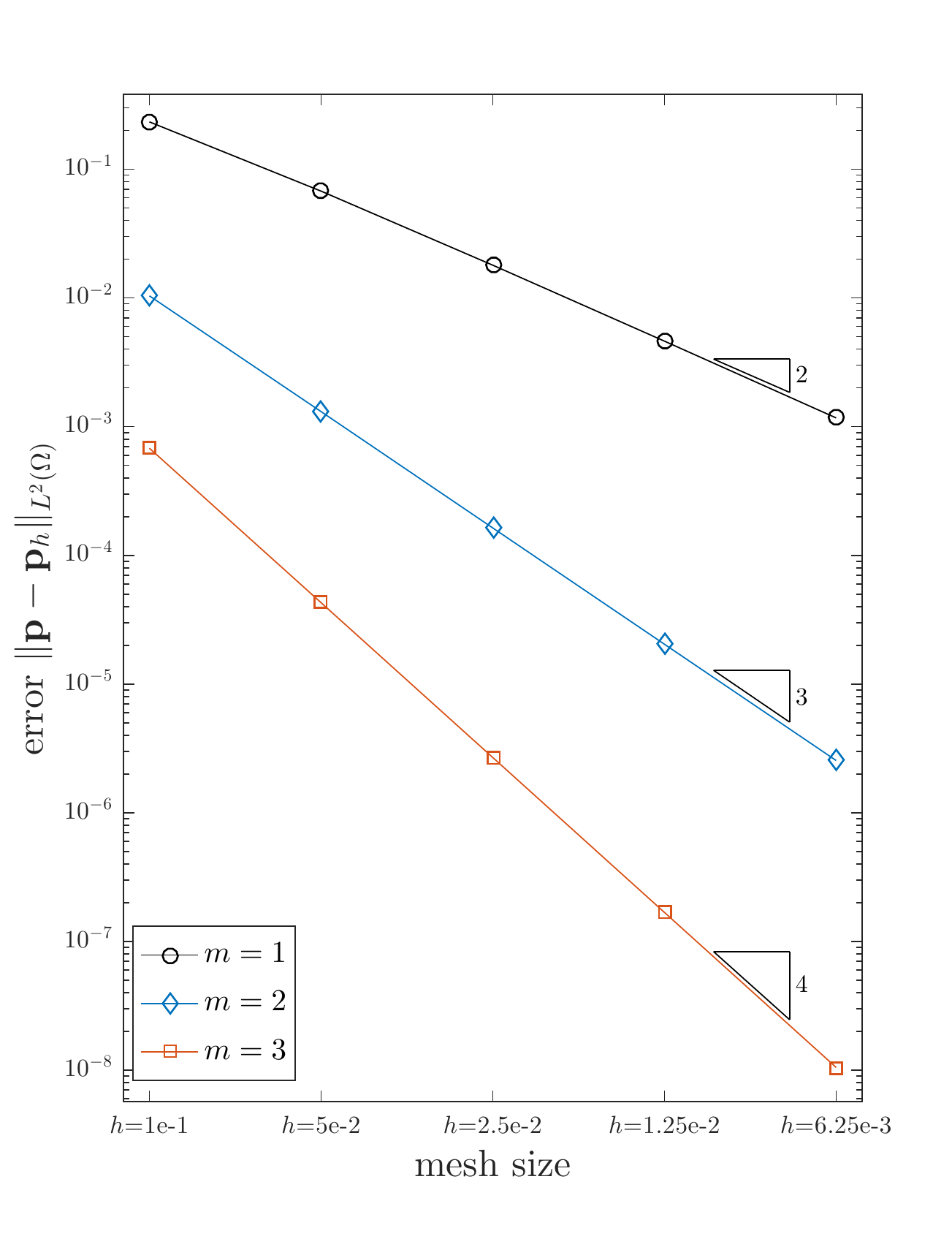}
  \caption{Example 2. The convergence rates of $\pnorm{\bm{p} -
  \bm{p}_h}$ (left) / $\| \bm{p} - \bm{p}_h\|_{L^2(\Omega)}$ (right).
  }
  \label{fig:ex2perror}
\end{figure}

\begin{figure}[htb]
  \centering
  \includegraphics[width=0.4\textwidth]{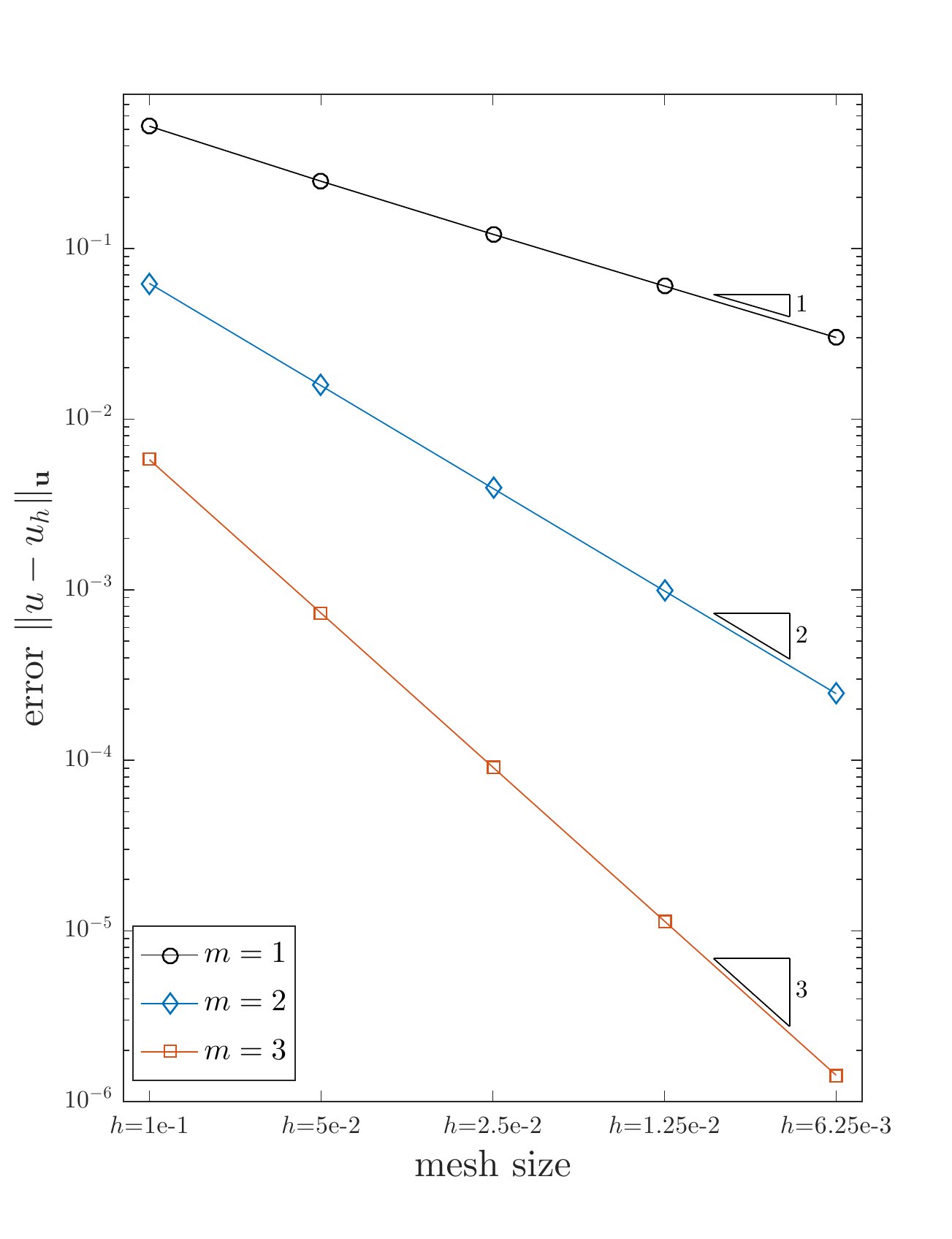}
  \includegraphics[width=0.4\textwidth]{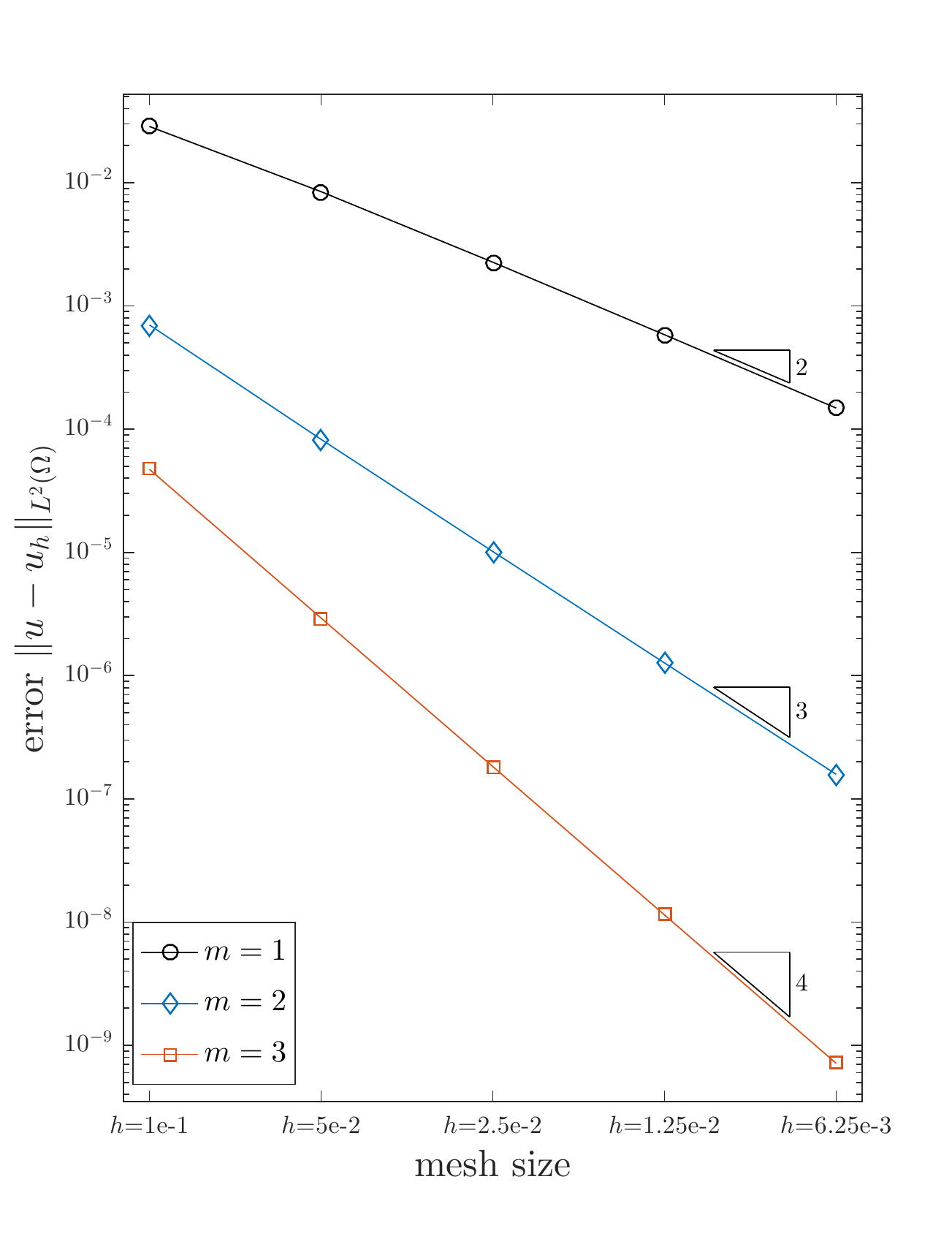}
  \caption{Example 2. The convergence rates of $\unorm{u - u_h}$
  (left) / $\| u - u_h\|_{L^2(\Omega)}$ (right).  }
  \label{fig:ex2uerror}
\end{figure}

\noindent \textbf{Example 3.} This is a 3D example and we solve a
problem in the unit cube $\Omega = (-1, 1)^3$. We partition the domain
$\Omega$ into a series of tetrahedral meshes with mesh size = $1/4$,
$1/8$, $1/16$, $1/32$, see Fig~\ref{fig:partition} for the tetrahedral
mesh with $h = 1/4$. The analytical solution and the coefficient
matrix are setup as
\begin{displaymath}
  u(x, y, z) = \cos(2\pi x) \cos(2\pi y) \cos(2\pi z),
\end{displaymath}
and 
\begin{displaymath}
  A(x, y, z) = \begin{bmatrix}
    10 & \frac{xy}{|xy|} & \frac{xz}{|xz|} \\
    \frac{yx}{|yx|} & 10 & \frac{yz}{|yz|} \\
    \frac{zx}{|zx|} & \frac{zy}{|zy|} & 10 \\
  \end{bmatrix},
\end{displaymath}
and the boundary condition $g$ and the source term $f$ are taken
suitably. We also use the finite element spaces $\bmr{S}_h^m \times
\wt{V}_h^m$ with $1 \leq m \leq 3$ to approximate $\bm{p}$ and $u$,
respectively. The numerical results are shown in
Fig~\ref{fig:ex3perror} and Fig~\ref{fig:ex3uerror}. 
%We note that for
%the case $m = 3$ our method shows the optimal convergence rates for
%all measurements although Lemma \ref{le:discreteMT} has a restriction
%$m \leq 2$. Besides, 
All computed convergence orders agree with the
theoretical results. 

\begin{figure}[htb]
  \centering
  \includegraphics[width=0.4\textwidth]{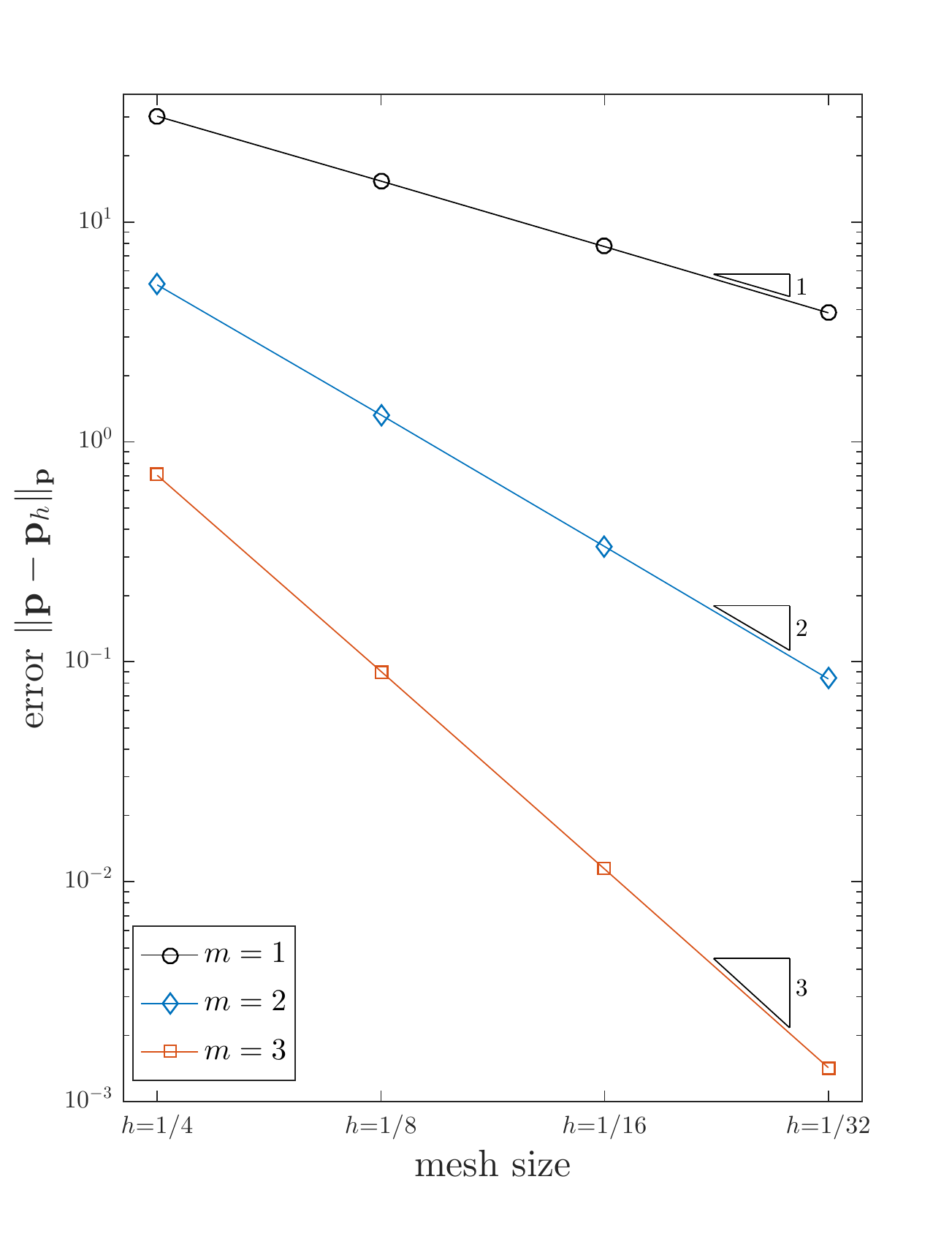}
  \includegraphics[width=0.4\textwidth]{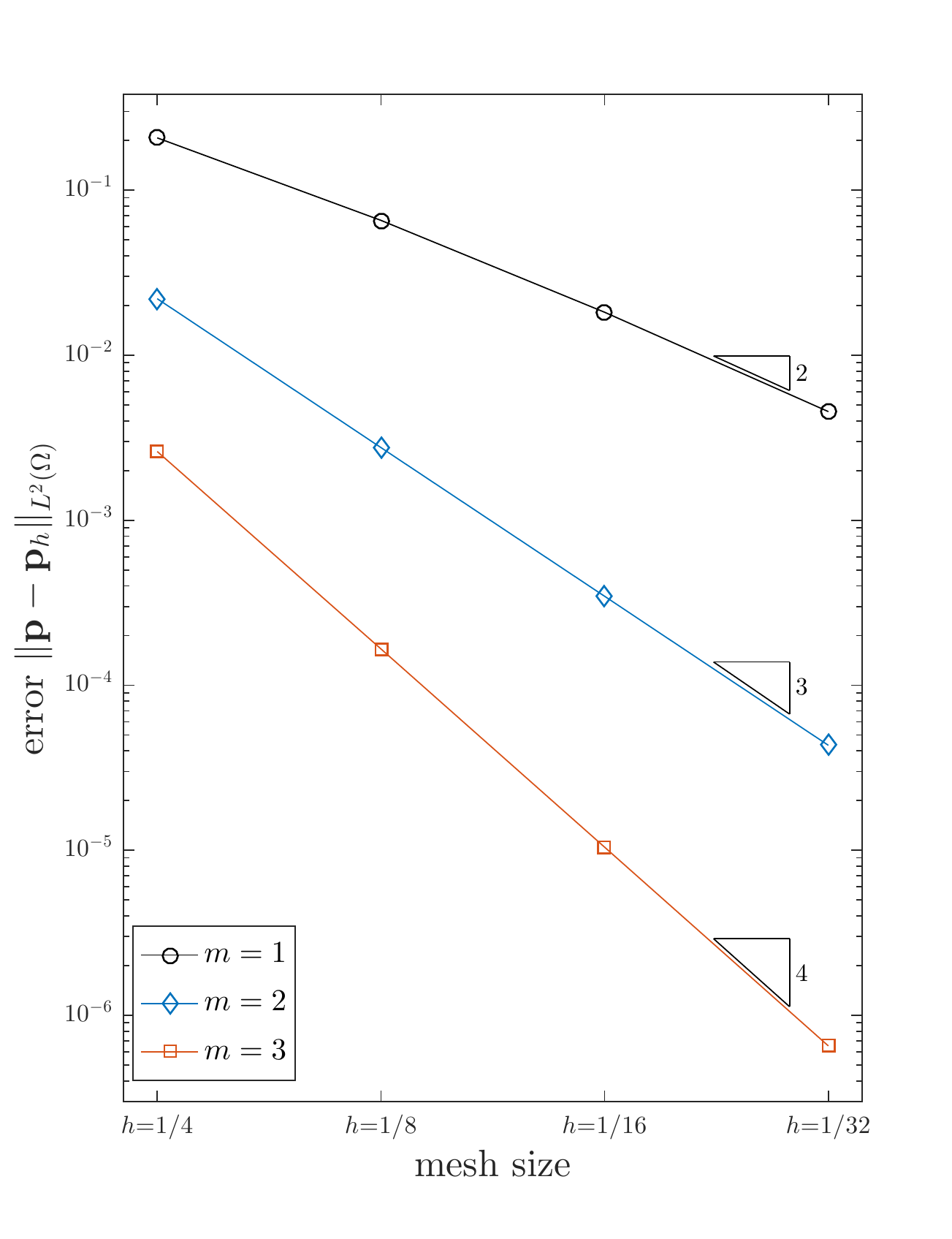}
  \caption{Example 3. The convergence rates of $\pnorm{\bm{p} -
  \bm{p}_h}$ (left) / $\| \bm{p} - \bm{p}_h\|_{L^2(\Omega)}$ (right).
  }
  \label{fig:ex3perror}
\end{figure}

\begin{figure}[htb]
  \centering
  \includegraphics[width=0.4\textwidth]{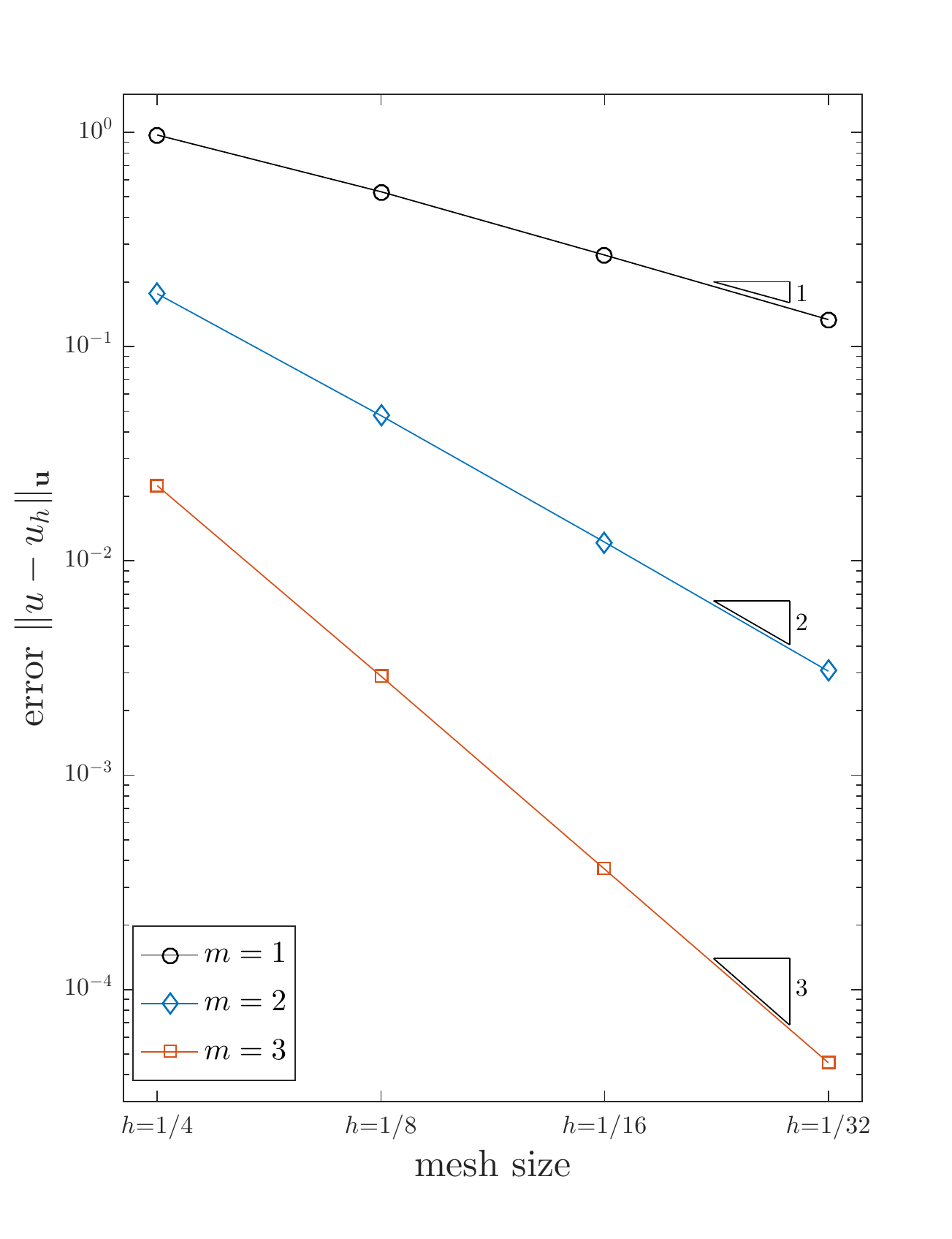}
  \includegraphics[width=0.4\textwidth]{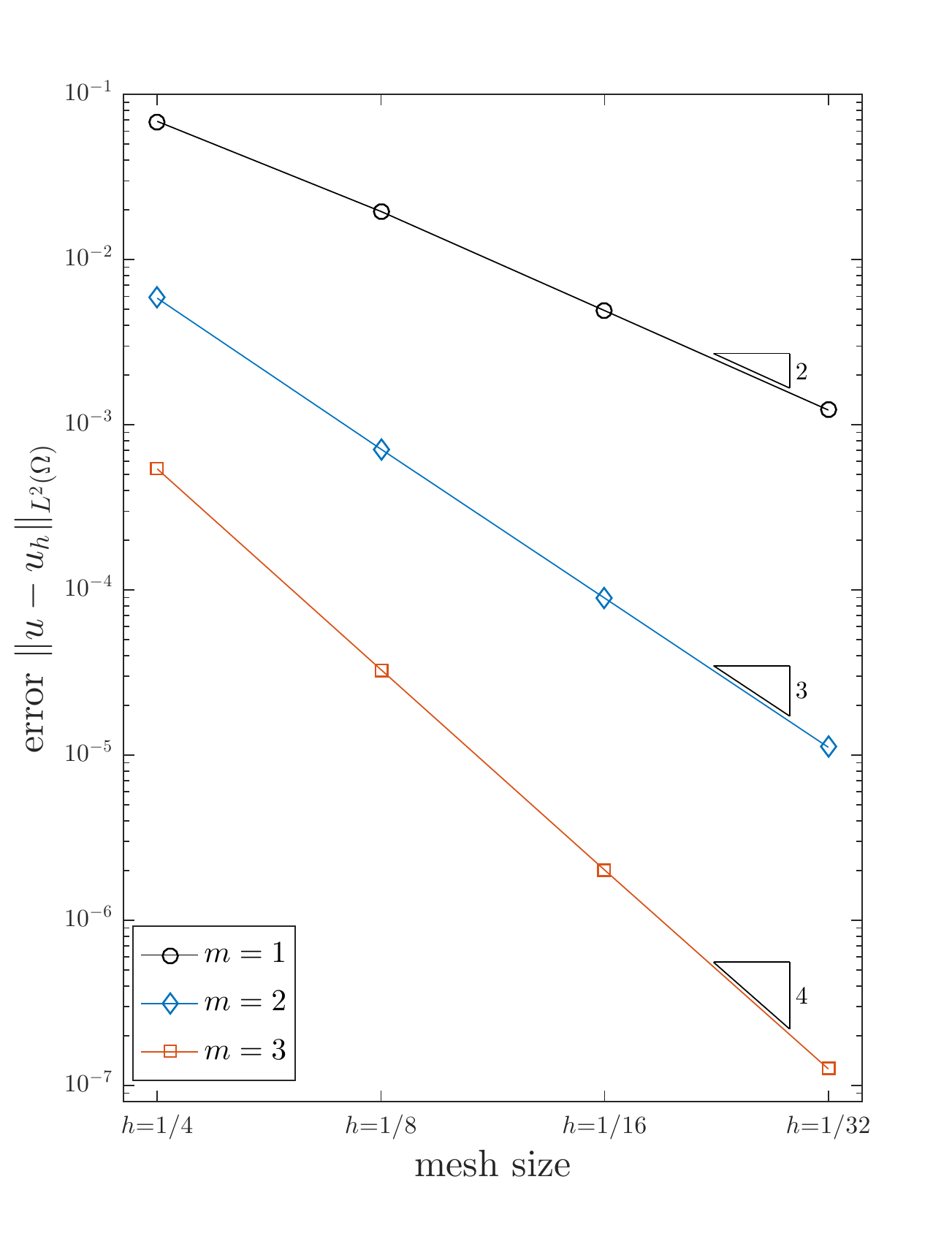}
  \caption{Example 3. The convergence rates of $\unorm{u - u_h}$
  (left) / $\| u - u_h\|_{L^2(\Omega)}$ (right).  }
  \label{fig:ex3uerror}
\end{figure}

\noindent \textbf{Example 4.} In this example, we consider the
problem on the domain $(0, 1)^2$ and the exact solution is chosen to
be
\begin{displaymath}
  u(x, y) = |\bm{x}|^\alpha,
\end{displaymath}
where $\alpha$ is a positive constant. The coefficient matrix
$A(x, y)$ takes the form 
\begin{displaymath}
  A(x, y) = \begin{bmatrix}
    1 + x^2 & xy \\
    xy & 1 + y^2 \\
  \end{bmatrix},
\end{displaymath}
and the data function $f$ and $g$ are selected properly. Notice that
$u$ belongs to the space $H^{\alpha + 1 - \delta}(\Omega)$ for
arbitrary small $\delta$. In the following, we take $\alpha = 1.2$ to
test the adaptive algorithm proposed in the previous section. The
parameter $\theta$ is chosen $\theta = 0.4$ and we consider the linear
accuracy $\bmr{S}_h^1 \times \wt{V}_h^1$ in the approximation to the
variables $\bm{p}$ and $u$. The mesh size of initial triangular
partition is taken as $h = 0.1$, see left figure in
Fig~\ref{fig:partition}. The whole convergence history of the uniform
refinement and adaptive refinement is displayed in
Fig~\ref{fig:ex4error}. For the uniform refinement, we observe the
error $\pnorm{\bm{p} - \bm{p}_h}$ decreases to zero at the speed
$O(h^{0.2})$, which agrees with the convergence analysis. For the
error $\| \bm{p} - \bm{p}_h\|_{L^2(\Omega)}$, the uniform refinement
leads to a reduced convergence speed $O(h^1)$. The reason may be
traced to the singularity of $u$ at the corner. Furthermore, for $u$
the errors $\unorm{u - u_h}$ and $\| u - u_h\|_{L^2(\Omega)}$ approach
to zero at the rate $O(h^1)$, which matches with the theoretical
analysis that the convergence rates in both norms for $u$ depend on
the convergence rate of $\| \bm{p} - \bm{p}_h\|_{L^2(\Omega)}$. For
the adaptive refinement, we note that all error measurements seem to
be optimal. The triangular meshes after 6 adaptive refinement steps
are shown in Fig~\ref{fig:refinement}. Clearly, the refinement is
pronounced in the regions where the solution is of low regularity.

\begin{figure}[htb]
  \centering
  \includegraphics[width=0.4\textwidth]{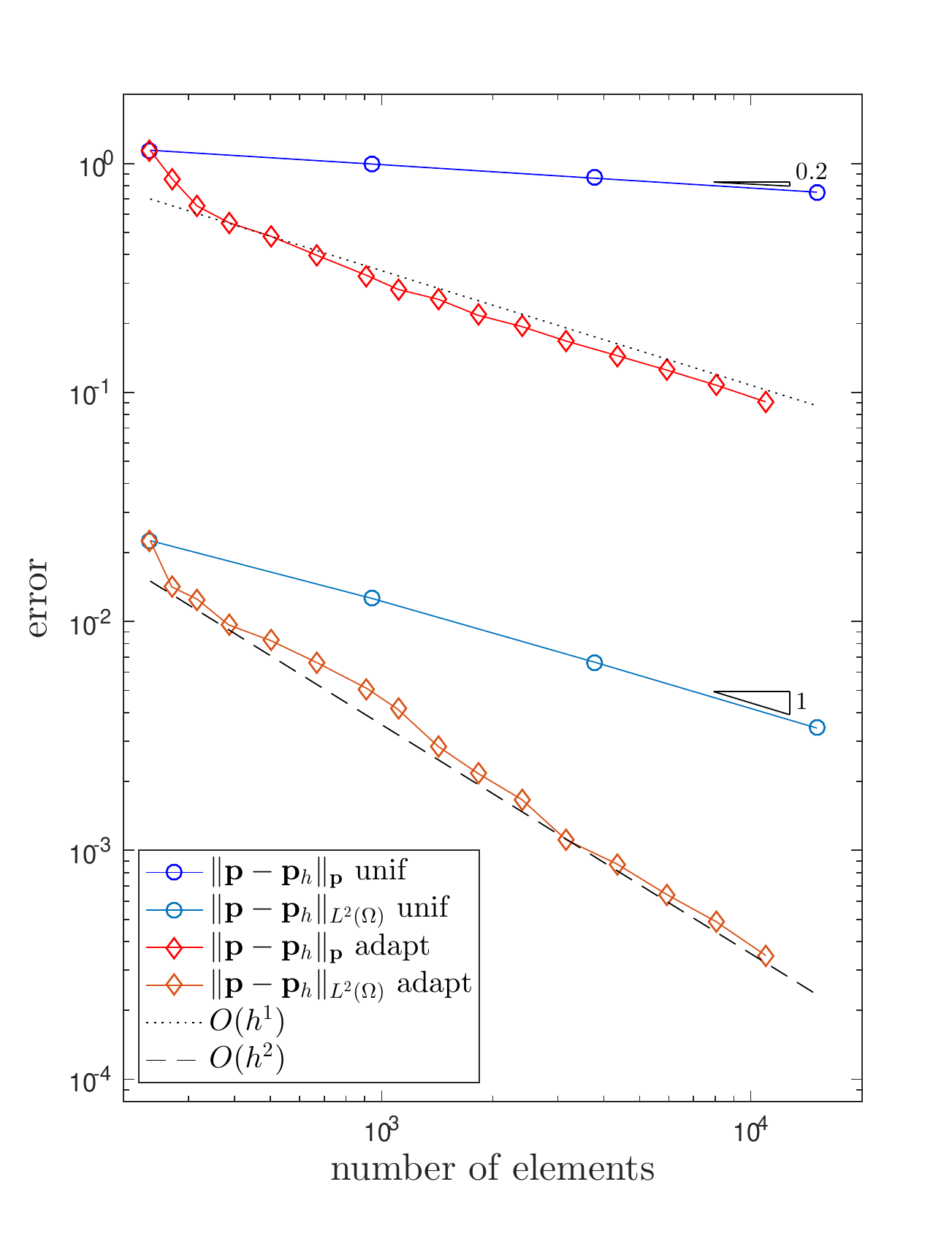}
  \includegraphics[width=0.4\textwidth]{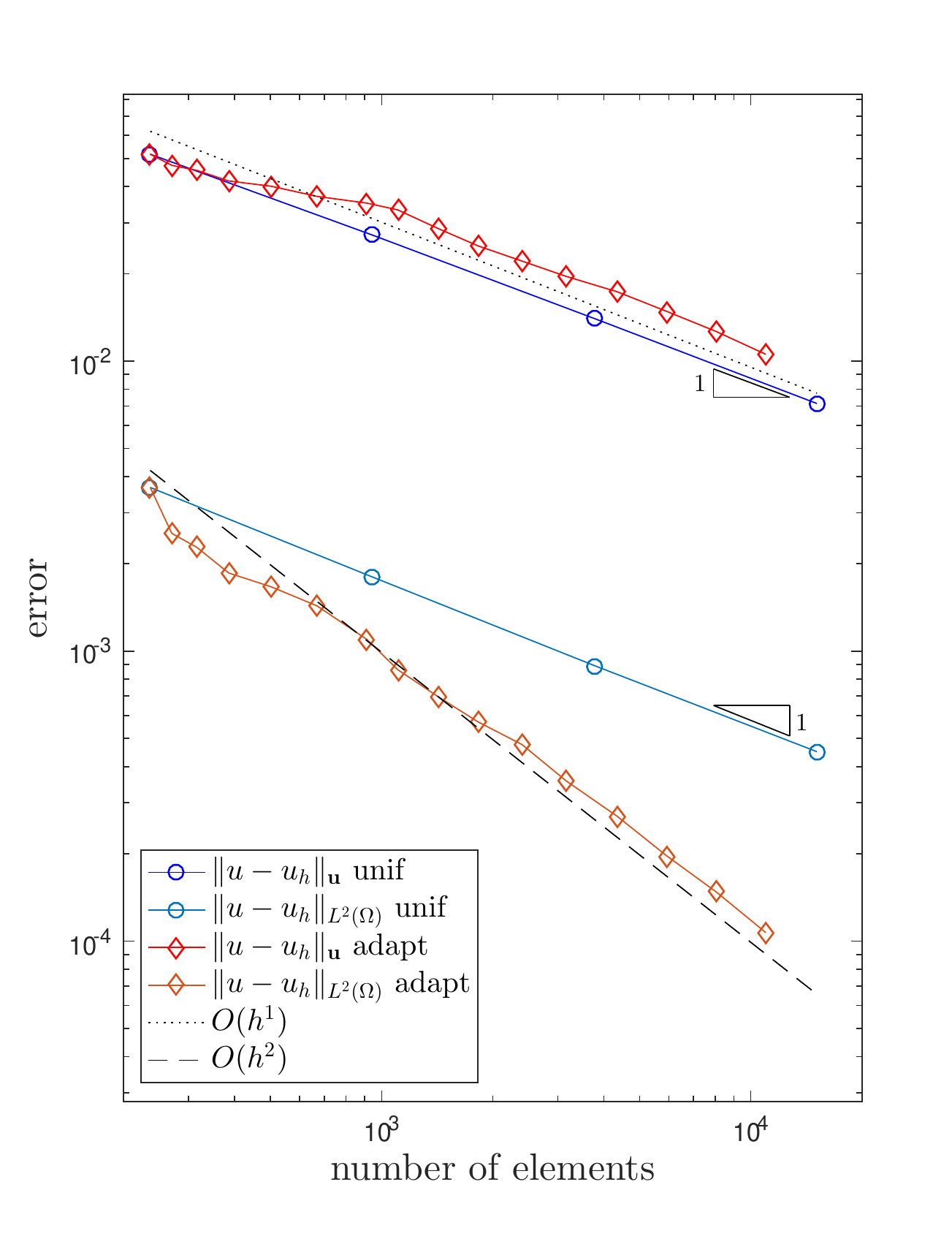}
  \caption{Example 4. The convergence history of $\bm{p}$ (left) / $u$
  (right).}
  \label{fig:ex4error}
\end{figure}

\begin{figure}
  \centering
  \includegraphics[width=0.4\textwidth]{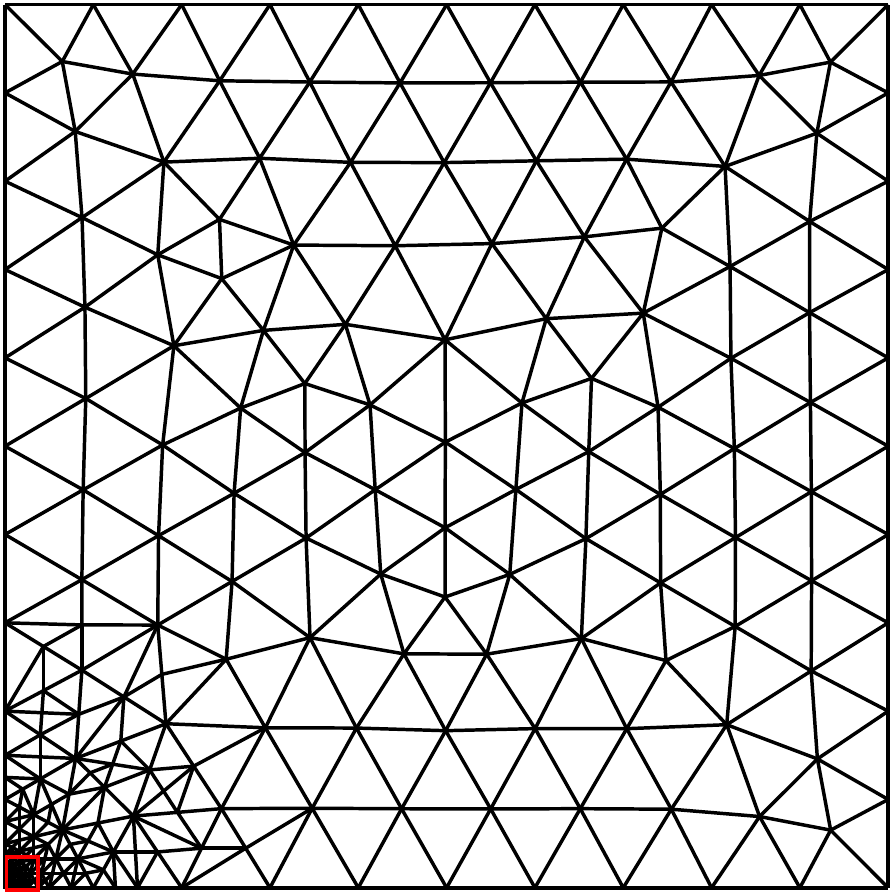}
  \hspace{25pt}
  \includegraphics[width=0.4\textwidth]{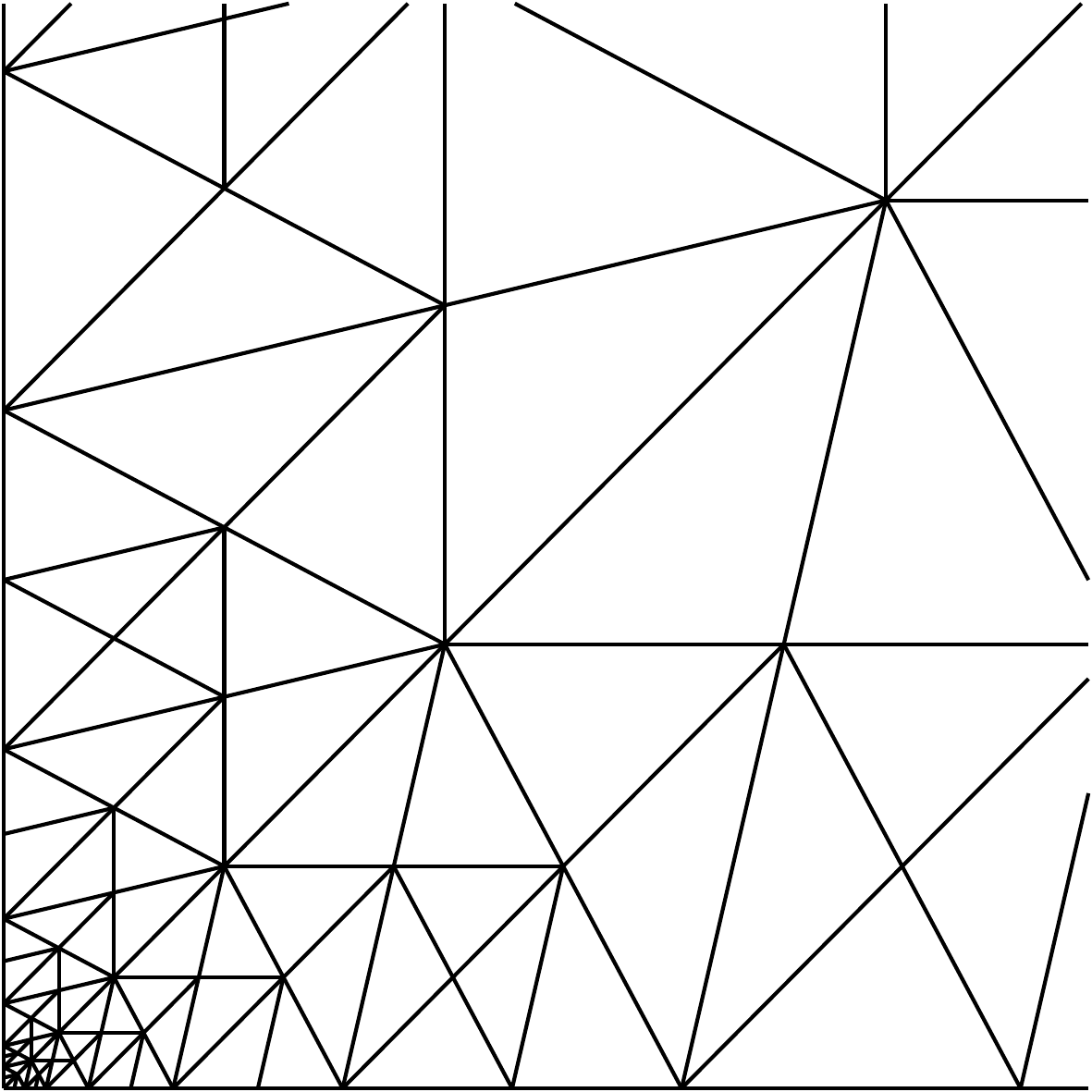}
  \caption{Triangular mesh after 6 adaptive refinement steps (left) /
  elements in the red rectangle (right).}
  \label{fig:refinement}
\end{figure}

%%% Local Variables:
%%% mode: latex
%%% TeX-master: "nondivergence"
%%% End:

% vim:spell:tw=70:fo+=Mn:cc=70
\section{Conclusion}
We proposed a sequential least squares finite element method for
elliptic equations in non-divergence form. We employed a novel
piecewise curl-free approximate space to solve the gradient variable
first and then we solve the primitive variable in the $C^0$ finite
element space. We proved the convergence rates for both variables with
respect to the $L^2$ norm and the energy norm. Optimal convergence
orders for all measurements were detected in numerical experiments. We
also tried an adaptive algorithm using $h$-adaptive method to improve
numerical efficiency for a problem of low regularity.

\section*{Acknowledgements}
This research was supported by the Science Challenge Project (No.
TZ2016002) and the National Science Foundation in China (No.
11971041).

%%% Local Variables:
%%% mode: latex
%%% TeX-master: "nondivergence"
%%% End:

\bibliographystyle{amsplain}
\bibliography{../ref}

\end{document}